\documentclass[12pt,a4paper]{article}
\usepackage{amsfonts}

\usepackage{mathrsfs}
\usepackage{xcolor}

\usepackage{stmaryrd}
\usepackage{amsmath}
\usepackage{amssymb}

\usepackage{bbm}
\usepackage{mathrsfs}
\usepackage{graphicx}
\usepackage{lipsum}

\usepackage{enumerate}
\usepackage{theorem}
\usepackage{indentfirst}
\makeatletter
\def\tank#1{\protected@xdef\@thanks{\@thanks
        \protect\footnotetext[0]{#1}}}
\def\bigfoot{

    \@footnotetext}
\makeatother

\newcommand{\ea}{\end{array}}

\newtheorem{theorem}{Theorem}[section]
\newtheorem{hypothesis}{Hypothesis}[section]
\newtheorem{proposition}{Proposition}[section]

\newtheorem{lemma}{Lemma}[section]
\newtheorem{definition}{Definition}[section]
\newtheorem{Rem}{Remark}[section]

{\theorembodyfont{\rmfamily}
}

\newenvironment{proof}{Proof.}

\def\EE{\mathbb{E}}

\def\vare{{\varepsilon}}
\def \eref#1{\hbox{(\ref{#1})}}

\def\EE{\mathbb{ E}}

\def\dt{d t}

\begin{document}
\title{{\Large \bf Strong Convergence Rates in Averaging Principle for Slow-Fast McKean-Vlasov SPDEs }\footnote{E-mail addresses: weihong@tju.edu.cn~(W.~H.), shihuli@jsnu.edu.cn~(S.~L.), weiliu@jsnu.edu.cn~(W. L.)}}

\author{{Wei Hong$^{a}$},~~{Shihu Li$^{b}$}{\footnote{Corresponding author}},~~{Wei Liu$^{b,c}$}
\\
 \small $a.$ Center for Applied Mathematics, Tianjin University, Tianjin 300072, China \\
 \small $b.$ School of Mathematics and Statistics, Jiangsu Normal University, Xuzhou 221116, China\\
  \small $c.$ Research Institute of Mathematical Sciences, Jiangsu Normal University, Xuzhou 221116, China }
\date{}
\maketitle
\begin{center}
\begin{minipage}{145mm}
{\bf Abstract.} In this paper, we aim to study the asymptotic behaviour for a class of McKean-Vlasov stochastic partial differential equations with slow and fast time-scales. Using the variational approach and classical Khasminskii time discretization, we show that the slow component strongly converges to the solution of the associated averaged equation. In particular, the corresponding convergence rates are also obtained. The main
results can be applied to demonstrate the averaging principle for various McKean-Vlasov nonlinear SPDEs such as stochastic porous media type equation, stochastic $p$-Laplace type equation and also some McKean-Vlasov stochastic differential equations.

\vspace{3mm} {\bf Keywords:}~SPDE;~Distribution dependence;~Averaging principle;~Convergence rate;~Porous media equation;~$p$-Laplace equation.


\end{minipage}
\end{center}

\section{Introduction}
The McKean-Vlasov SDEs, also called mean-field SDEs or distribution dependent SDEs, have attracted much attention  in recent years, which was initiated by McKean \cite{M}. Roughly speaking, these are SDEs where their coefficients also depend on the distribution of solutions.
This type of models can be used to characterize the limiting behaviors of $N$-interacting particle systems of mean-field type while $N$ goes to infinity (also called propagation of chaos), one can see \cite{M1} for more background on this topic.
The main motivation for studying the McKean-Vlasov SDEs is due to its wide applications since the evolution of stochastic systems often rely on both the microcosmic position and the macrocosmic distribution of the particles.
Furthermore, the McKean-Vlasov SDEs also have some intrinsic link with the nonlinear Fokker-Planck-Kolmogorov equations (cf.~\cite{BR1,HRW}). More precisely, the corresponding distribution density (denoted by $\rho_t$) of solutions to McKean-Vlasov SDEs solves the following PDE
$$\partial_t\rho_t=L^*\rho_t,~~t\geq 0,$$
where $L$ is a second order differential operator and $L^*$ denotes its adjoint operator.

McKean-Vlasov S(P)DEs  have been extensively investigated in recent years.~For instance, Wang \cite{W1} proved the strong and weak existence and uniqueness of solutions to McKean-Vlasov monotone SDEs, and also studied the corresponding exponential ergodicity and Harnack type inequality under some strongly dissipative conditions, which are applicable to e.g. the homogeneous Landau equations. After that, Zhang \cite{Z} investigated the weak solutions of McKean-Vlasov SDEs with singular coefficients, which can be used to characterize the existence of weak solutions to 2D Navier-Stokes equations
with measure as initial vorticity.
Recently, the authors \cite{HL} used the generalized variational framework  to study the existence of unique strong solution for a class of distribution dependent stochastic porous media equation. Barbu and R\"{o}ckner \cite{BR1} also used the nonlinear Fokker-Planck equations to investigate some McKean-Vlasov SDEs. We refer the interested reader to \cite{BR2,BLPR,H,HRZ,RTW} and references therein for more recent results on this topic.

In this paper, we will consider the following slow-fast McKean-Vlasov  stochastic partial differential equations
\begin{equation}\label{e00}
\left\{ \begin{aligned}
&dX^{\varepsilon}_t=\left[A_1(X^{\varepsilon}_t,\mathscr{L}_{X^{\varepsilon}_t})
+f(X^{\varepsilon}_t,\mathscr{L}_{X^{\varepsilon}_t},Y^{\varepsilon}_t)\right]dt
+B_1(X^{\varepsilon}_t,\mathscr{L}_{X^{\varepsilon}_t})dW^1_t,\\
&dY^{\varepsilon}_t=\frac{1}{\varepsilon}A_2(X^{\varepsilon}_t,\mathscr{L}_{X^{\varepsilon}_t},Y^{\varepsilon}_t)dt
+\frac{1}{\sqrt{\varepsilon}}B_2(X^{\varepsilon}_t,\mathscr{L}_{X^{\varepsilon}_t},Y^{\varepsilon}_t)d W^{2}_{t},
\\&X^{\varepsilon}_0=x, Y^{\varepsilon}_0=y,
\end{aligned} \right.
\end{equation}
where $\{W^{i}_t\}_{t\in [0,T]}$, $i=1,2$, are independent cylindrical Wiener processes defined on a complete filtered probability space $\left(\Omega,\mathscr{F},\mathscr{F}_{t\geq0},\mathbb{P}\right)$, $\mathscr{L}_{X^{\varepsilon}_t}$ denotes the law of $X^{\varepsilon}_t$, $\varepsilon$ is a small and
positive parameter describing the ratio of time-scale between processes $X^{\varepsilon}_t$ and $Y^{\varepsilon}_t$. With this time-scale, the variable $X^{\varepsilon}_t$ is referred to as the slow component and $Y^{\varepsilon}_t$ is referred to as the fast component. Multiscale systems are very common in many fields of sciences, like material sciences,
fluids dynamics, climate dynamics, etc. For example, dynamics of chemical reaction networks often take place on notably different time-scales,  from the order of nanoseconds to the order of several days,  the reader can see \cite{BR,WE,HKW,MCCTB} and the references therein for more precise background and applications.

One natural question is what will happen to the solution of the system \eref{e00} as $\varepsilon \rightarrow0$? This question arises naturally from both physical and mathematical standpoints. Averaging principle is a powerful tool for some qualitative analysis of stochastic dynamical systems with different time-scales. The averaging principle for stochastic dynamical systems  with fast
and slow time-scales can be viewed as a law of large numbers, in the cases where a slow
component is driven by an equation with coefficients depending on a fast component, which is an
ergodic stochastic process: when the separation of time-scales goes to infinity, the slow component
converges to the solution of an averaged equation whose coefficients have been averaged out with
respect to some invariant probability distribution for the fast component.

Apart from the above motivations, the averaging principle itself is also theoretically interesting, which
has been studied a lot in the literatures. The averaging principle for dynamical systems with different time-scales  was first studied by
Bogoliubov and Mitropolsky \cite{BM} for
the deterministic systems, afterwards Khasminskii \cite{K1} developed the averaging principles for stochastic dynamical systems, see e.g.~\cite{G,HL1,L1} for further generalizations on different types of SDEs.
Recently, the averaging principles for SPDEs  have also been intensively investigated in the literature.
For example, Dong et al.~\cite{DSXZ} studied the strong and weak averaging principle for stochastic Burgers equations,
Br\'{e}hier \cite{Br1,Br2} gave the strong and weak orders in averaging for stochastic evolution equation of parabolic type with slow and fast time-scales.
The averaging principle for the nonautonomous slow-fast systems of stochastic reaction-diffusion equations was considered in \cite{CL}.
Moreover, Liu et al.~\cite{LRSX1} also established the strong averaging principle for a class of SPDEs with locally monotone coefficients.
For more results on this subject, we refer to \cite{BSYY,BYY,CF,FWL,FWLL,GP2,GP3,PXW,RX2,WR12,WRD12,XPW} and the references therein.

However, to the best of our knowledge, there is no result concerning the averaging principle for McKean-Vlasov type SPDEs in the literature so far. Recently, based on the techniques of time discretization and Poisson equation, R\"{o}ckner et al.~\cite{RSX} established the strong convergence rates of averaging principle for McKean-Vlasov SDEs with global Lipschitz coefficients. Bezemek and Spiliopoulos \cite{BS} also studied the large deviations principle for interacting particle systems of diffusion type in multiscale environments. Note that the above results are  for the finite dimensional SDE case. In this paper, we aim to study
the strong  averaging principle for a class of McKean-Vlasov (nonlinear) SPDEs with slow and fast time-scales. More precisely,
under some appropriate assumptions, we shall prove that
\begin{equation}\label{1.2}
\mathbb{E} \left( \sup_{t\in[0,T]} \|X^{\varepsilon}_t-\bar{X}_t \|_{H_1}^{2} \right)\leq C\varepsilon^{1/3}\rightarrow0,~~~\text{as}~\varepsilon\rightarrow0,
\end{equation}
where $\bar{X}_t$ is the solution of the averaged equation (see equation \eref{1.3} below). In particular, the corresponding convergence rate of (\ref{1.2}) is also derived, which is very important in some applications. For instance, the rate
of convergence is crucial for the analysis of numerical schemes used to approximate the slow
component $X^{\varepsilon}$.

In the distribution-independent case, the convergence rates for two-time-scale SDEs have been studied in some works, see e.g.~\cite{G,LD,RX2,RSX1} and the references therein. Note that there are only few results concerning the strong convergence rates for SPDEs in the literature. Fu et al.~\cite{FWL}  established the convergence rate of order $1/4$ for a class of stochastic hyperbolic-parabolic equations. Dong et al.~\cite{DSXZ} also studied the strong convergence of stochastic Burgers equations with some Logarithmic convergence order. An important development concerning strong convergence rate for SPDEs was established by Br\'{e}hier \cite{Br2} with the convergence rate of order $1/2$, which
is the optimal order of strong convergence in general. However, most papers in the literature investigated strong convergence rate using the mild solution approach, which is
only applicable to some semilinear SPDEs. In this paper, we establish the convergence rate of order $1/6$ for a class of Mckean-Vlasov quasilinear SPDEs.  In \cite{Br2}, to obtain the optimal convergence order, some fairly strong conditions  such as the regularity of second and higher order derivatives of the coefficients and more regular initial value are assumed. The convergence rate obtained here might not be optimal, since we only assume the coefficients satisfy some monotonicity and coercivity conditions, which is in general much weaker than the assumptions in \cite{Br2}.
As examples, our main results are applicable to some Mckean-Vlasov quasilinear SPDEs such as distribution dependent stochastic porous media type equations, stochastic $p$-Laplace type equations, which are also new in the distribution-independent case.

It should be mentioned that this is the first averaging principle result for two-time-scale McKean-Vlasov
(nonlinear) SPDEs. In addition, We also remark that there are some merits to analyze nonlinear operators (even linear operators) on a Gelfand triple replacing a single space, which helps us to deal with the McKean-Vlasov type SPDEs with nonlinear terms (cf. e.g. \cite{EO,KR}). Since the well-posedness of the two-time-scale McKean-Vlasov SPDEs \eref{e00} is not covered by the classical theory of monotone SPDEs (\cite{LR1}) and the McKean-Vlasov case (\cite{H,HL}),
based on the technique of Galerkin type approximation and monotonicity arguments, we first prove the existence and uniqueness of variational solutions for the two-time-scale McKean-Vlasov SPDEs. Then we aim to investigate the strong averaging principle for this type of models. The proof here is mainly inspired by the well-known time discretization method, which was first developed by Khasminskii in \cite{K1} for finite dimensional dynamical systems under random influences.
We need to point out that compared with the results of McKean-Vlasov SDEs in \cite{RSX}, in order to cover some infinite dimensional nonlinear SPDE models, we now consider the system in two Gelfand triples, thus we have to derive some apriori estimates of solutions involving different spaces and overcome some non-trivial difficulties caused by the nonlinear terms, which is quite different to the finite dimensional case.

The remainder of this manuscript is organized as follows. In section 2, we construct the variational framework for a class of McKean-Vlasov SPDEs and give the main results of the present paper. In section 3, we show the existence and uniqueness of solutions to the system \eref{e00}. In section 4, we devote to proving the averaging principle for the system \eref{e00}, and in section 5 some concrete McKean-Vlasov SPDE models are given to illustrate the applications of the main results.

\section{Main Results}
\setcounter{equation}{0}
 \setcounter{definition}{0}

Let us denote by $(U_i,\langle\cdot,\cdot\rangle_{U_i})$ and $(H_i, \langle\cdot,\cdot\rangle_{H_i})$, $i=1,2$, some separable Hilbert spaces, and $H_i^*$ the dual space of $H_i$. Let $V_i$, $i=1,2$, denote the reflexive Banach spaces such that the embedding $V_i\subset H_i$ is continuous and dense. We identify $H_i$ with its dual space according to the Riesz isomorphism, which gives the following Gelfand triples
$$V_i\subset H_i(\cong H_i^*)\subset V_i^*.$$
The dualization between spaces $V_i$ and $V_i^*$ is denoted by $_{V_i^*}\langle\cdot,\cdot\rangle_{V_i}$. It is obvious that $$_{V_i^*}\langle\cdot,\cdot\rangle_{V_i}|_{{H_i}\times{V_i}}=\langle\cdot,\cdot\rangle_{H_i},~i=1,2.$$
Let $L_2(U_i,H_i)$ be the space of all Hilbert-Schmidt operators from $U_i$ to $H_i$.

Denote by $\mathscr{P}(H_1)$ the space of all probability measures on $H_1$ equipped with the weak topology. Now we define
$$\mathscr{P}_2(H_1):=\Big\{\mu\in\mathscr{P}(H_1):\mu(\|\cdot\|_{H_1}^2):=
\int_{H_1}\|\xi\|_{H_1}^2\mu(d\xi)<\infty\Big\}.$$
Then $\mathscr{P}_2(H_1)$ is a Polish space under the following $L^2$-Wasserstein metric
$$\mathbb{W}_{2,H_1}(\mu,\nu):=\inf_{\pi\in\mathscr{C}(\mu,\nu)}\Big
(\int_{H_1\times H_1}\|\xi-\eta\|_{H_1}^2\pi(d\xi,d\eta)\Big)^{\frac{1}{2}},
~\mu,\nu\in\mathscr{P}_2(H_1),$$
here $\mathscr{C}(\mu,\nu)$ stands for the set of all couplings for the measures $\mu$ and $\nu$, i.e., $\pi\in\mathscr{C}(\mu,\nu)$ is a probability measure on $H_1\times H_1$ such that $\pi(\cdot\times H_1)=\mu$ and $\pi(H_1\times \cdot)=\nu$.

For some measurable maps
$$
A_1:V_1\times\mathscr{P}_2(H_1)\rightarrow V_1^*,~~f:H_1\times\mathscr{P}_2(H_1)\times H_2\to H_1,~~B_1:V_1\times\mathscr{P}_2(H_1)\to L_2(U_1,H_1),
$$
and
$$
A_2:H_1\times\mathscr{P}_2(H_1)\times V_2\rightarrow V_1^*,~~B_2:H_1\times\mathscr{P}_2(H_1)\times V_2\to L_2(U_2,H_2),
$$
we consider the following two-time-scale McKean-Vlasov SPDEs
\begin{equation}\label{e1}
\left\{ \begin{aligned}
&dX^{\varepsilon}_t=\left[A_1(X^{\varepsilon}_t,\mathscr{L}_{X^{\varepsilon}_t})
+f(X^{\varepsilon}_t,\mathscr{L}_{X^{\varepsilon}_t},Y^{\varepsilon}_t)\right]dt
+B_1(X^{\varepsilon}_t,\mathscr{L}_{X^{\varepsilon}_t})dW^1_t,\\
&dY^{\varepsilon}_t=\frac{1}{\varepsilon}A_2(X^{\varepsilon}_t,\mathscr{L}_{X^{\varepsilon}_t},Y^{\varepsilon}_t)dt
+\frac{1}{\sqrt{\varepsilon}}B_2(X^{\varepsilon}_t,\mathscr{L}_{X^{\varepsilon}_t},Y^{\varepsilon}_t)d W^{2}_{t},
\\&X^{\varepsilon}_0=x,~Y^{\varepsilon}_0=y,
\end{aligned} \right.
\end{equation}
where $\{W^{i}_t\}_{t\in [0,T]}$, $i=1,2$, are $U_i$-valued independent cylindrical Wiener process defined on a complete filtered probability space $\left(\Omega,\mathscr{F},\mathscr{F}_{t\geq0},\mathbb{P}\right)$, initial values $x,y$ belong to $H_1,H_2$ respectively.

We first assume that the coefficients in \eref{e1} satisfy the following two hypothesises.
\begin{hypothesis}\label{h1}
There are some constants $\alpha\geq2$, $\theta>0$ and $c_1>0$ such that for all $u,v\in V_1$, $u_1,u_2\in H_1$, $v_1,v_2\in H_2$ and $\mu,\nu\in\mathscr{P}_2(H_1)$ we have
\begin{enumerate}
\item [$({\mathbf{A}}{\mathbf{1}})$](Demicontinuity) The map
\begin{eqnarray*}
V_1\times\mathscr{P}_2(H_1)\ni(u,\mu)\mapsto{_{V_1^*}\langle A_1(t,u,\mu),v\rangle_{V_1}}
\end{eqnarray*}
is continuous.

\item [$({\mathbf{A}}{\mathbf{2}})$](Monotonicity and Lipschitz)
\begin{eqnarray}\label{m}
_{{V_1}^*}\langle A_1(u,\mu)-A_1(v,\nu),u-v\rangle_{V_1} \leq c_1\big(\|u-v\|_{H_1}^2+\mathbb{W}_{2,H_1}(\mu,\nu)^2\big).
\end{eqnarray}
Moreover,
\begin{eqnarray*}
\!\!\!\!\!\!\!\!&&\|f(u_1,\mu,v_1)-f(u_2,\nu,v_2)\|_{H_1}
\nonumber\\
\leq \!\!\!\!\!\!\!\!&& c_1\big(\|u_1-u_2\|_{H_1}+\|v_1-v_2\|_{H_2}+\mathbb{W}_{2,H_1}(\mu,\nu)\big)
\end{eqnarray*}
and
\begin{eqnarray*}
\|B_1(u,\mu)-B_1(v,\nu)\|_{L_2(U_1,H_1)}
\leq \!\!\!\!\!\!\!\!&& c_1\big(\|u-v\|_{H_1}+\mathbb{W}_{2,H_1}(\mu,\nu)\big).
\end{eqnarray*}
\item [$({\mathbf{A}}{\mathbf{3}})$](Coercivity)
\begin{eqnarray*}
2_{V_1^*}\langle A_1(u,\mu),u\rangle_{V_1}+\|B_1(u,\mu)\|_{L_2(U_1,H_1)}^2\leq -\theta\|u\|_{V_1}^\alpha+c_1\big(1+\|u\|_{H_1}^2+\mu(\|\cdot\|_{H_1}^2)\big).
\end{eqnarray*}
\item [$({\mathbf{A}}{\mathbf{4}})$](Growth)
\begin{eqnarray*}
\|A_1(u,\mu)\|_{{V_1}^*}^{\frac{\alpha}{\alpha-1}}\leq c_1\big(1+\|u\|_{V_1}^{\alpha}+\mu(\|\cdot\|_{H_1}^2)\big).
\end{eqnarray*}
\end{enumerate}
\end{hypothesis}

\begin{hypothesis}\label{h2}
There are some constants $\beta>1$, $\eta,\kappa>0$, $L_{B_2},c_2>0$ such that for all $u,u_1,u_2\in H_1,v,v_1,v_2,w\in V_2$ and $\mu,\nu\in\mathscr{P}_2(H_1)$ we have
\begin{enumerate}
\item [$({\mathbf{H}}{\mathbf{1}})$](Demicontinuity) The map
\begin{eqnarray*}
H_1\times\mathscr{P}_2(H_1)\times V_2\ni(u,\mu,v)\mapsto{_{V_2^*}\langle A_2(u,\mu,v),w\rangle_{V_2}}
\end{eqnarray*}
is continuous.

\item [$({\mathbf{H}}{\mathbf{2}})$](Monotonicity and Lipschitz)
\begin{eqnarray}\label{c1}
\!\!\!\!\!\!\!\!&&_{V_2^*}\langle A_2(u_1,\mu,v_1)-A_2(u_2,\nu,v_2),v_1-v_2\rangle_{V_2}
\nonumber\\
\leq\!\!\!\!\!\!\!\!&& -\kappa\|v_1-v_2\|_{H_2}^2+c_2\big(\|u_1-u_2\|_{H_1}^2+\mathbb{W}_{2,H_1}(\mu,\nu)^2\big)
\end{eqnarray}
and
\begin{eqnarray*}
\|B_2(u_1,\mu,v_1)-B_2(u_2,\nu,v_2)\|_{L_2(U,H_2)}\leq L_{B_2}\|v_1-v_2\|_{H_2}+c_2\big(\|u_1-u_2\|_{H_1}+\mathbb{W}_{2,H_1}(\mu,\nu)\big).
\end{eqnarray*}

\item [$({\mathbf{H}}{\mathbf{3}})$](Coercivity)
\begin{eqnarray*}
\!\!\!\!\!\!\!\!&&2_{V_2^*}\langle A_2(u,\mu,v),v\rangle_{V_2}+\|B_2(u,\mu,v)\|_{L_2(U_2,H_2)}^2
\nonumber\\
\leq \!\!\!\!\!\!\!\!&& c_2\left(1+\|v\|_{H_2}^2+\|u\|_{H_1}^2+\mu(\|\cdot\|_{H_1}^2)\right)-\eta\|v\|_{V_2}^\beta.
\end{eqnarray*}

\item [$({\mathbf{H}}{\mathbf{4}})$](Growth)
\begin{eqnarray*}
\|A_2(u,\mu,v)\|_{V_2^*}^{\frac{\beta}{\beta-1}}\leq c_2\big(1+\|v\|_{V_2}^{\beta}+\|u\|_{H_1}^{2}+\mu(\|\cdot\|_{H_1}^2)\big).
\end{eqnarray*}
\end{enumerate}
\end{hypothesis}
\begin{Rem}\label{r1}
(i) Note that the assumptions for the slow component of system (\ref{e1}) in Hypothesis \ref{h1} extends the classical variational framework to the distribution dependent case, which are applicable to various McKean-Vlasov quasilinear and semilinear SPDEs, such as distribution dependent stochastic porous media type equations and stochastic p-Laplace type equations.

(ii) The strictly monotone condition (\ref{c1}) is used to guarantee the existence and uniqueness of invariant probability measure and the associated exponential ergodicity for the frozen equation (see Eq.~(\ref{FEQ1}) below) of the fast component of system (\ref{e1}). A typical example satisfying  Hypothesis \ref{h2} will be presented in section \ref{sec5}.
\end{Rem}

The definition of variational solution to system~(\ref{e1}) is given as follows.
\begin{definition}\label{d1}
For any $\varepsilon>0$, we call a continuous $H_1\times H_2$-valued $(\mathscr{F}_t)_{t\geq 0}$-adapted process $(X^{\varepsilon}_t,Y^{\varepsilon}_t)_{t\in[0,T]}$ is a solution of the system (\ref{e1}), if for its $dt\times \mathbb{P}$-equivalent class $(\hat{X}^{\varepsilon}_t,\hat{Y}^{\varepsilon}_t)_{t\in[0,T]}$ satisfying
\begin{equation*}
\hat{X}^{\varepsilon}\in L^{\alpha}\big([0,T]\times\Omega,dt\times\mathbb{P};V_1\big)\cap L^2\big([0,T]\times\Omega,dt\times\mathbb{P};H_1\big),
\end{equation*}
\begin{equation*}
\hat{Y}^{\varepsilon}\in L^{\beta}\big([0,T]\times\Omega,dt\times\mathbb{P};V_2\big)\cap L^2\big([0,T]\times\Omega,dt\times\mathbb{P};H_2\big),
\end{equation*}
where $\alpha,\beta$ is the same as defined in $({\mathbf{A}}{\mathbf{3}})$ and $({\mathbf{H}}{\mathbf{3}})$, respectively, and $\mathbb{P}$-a.s.,
\begin{equation*}
\left\{ \begin{aligned}
&dX^{\varepsilon}_t=x+\int_0^t\left[A_1(\bar{X}^{\varepsilon}_s,\mathscr{L}_{\bar{X}^{\varepsilon}_s})
+f(\bar{X}^{\varepsilon}_s,\mathscr{L}_{\bar{X}^{\varepsilon}_s},Y^{\varepsilon}_s)\right]ds
+\int_0^tB_1(\bar{X}^{\varepsilon}_s,\mathscr{L}_{\bar{X}^{\varepsilon}_s})dW^1_s,\\
&dY^{\varepsilon}_t=y+\frac{1}{\varepsilon}\int_0^tA_2(\bar{X}^{\varepsilon}_s,\mathscr{L}_{\bar{X}^{\varepsilon}_s},\bar{Y}^{\varepsilon}_s)ds
+\frac{1}{\sqrt{\varepsilon}}\int_0^tB_2(\bar{X}^{\varepsilon}_s,\mathscr{L}_{\bar{X}^{\varepsilon}_s},\bar{Y}^{\varepsilon}_s)d W^{2}_{s},
\end{aligned} \right.
\end{equation*}
here $(\bar{X}^{\varepsilon},\bar{Y}^{\varepsilon})$ is an $V_1\times V_2$-valued progressively measurable $dt\times\mathbb{P}$-version of $(\hat{X}^{\varepsilon},\hat{Y}^{\varepsilon})$.
\end{definition}

The first result is about the existence and uniqueness of solutions to system (\ref{e1}).
\begin{theorem}\label{Th1}
Suppose that the assumptions $({\mathbf{A}}{\mathbf{1}})$-$({\mathbf{A}}{\mathbf{4}})$ and $({\mathbf{H}}{\mathbf{1}})$-$({\mathbf{H}}{\mathbf{4}})$ hold. For each $\varepsilon>0$ and initial values $x\in H_1$, $y\in H_2$, system~(\ref{e1}) has a unique solution $(X^{\varepsilon}_t,Y^{\varepsilon}_t)_{t\in[0,T]}$ in the sense of Definition \ref{d1}.
\end{theorem}

The next main result of this paper is the strong averaging principle for the system \eref{e1}.
\begin{theorem}\label{main result 1}
Suppose that the assumptions $({\mathbf{A}}{\mathbf{1}})$-$({\mathbf{A}}{\mathbf{4}})$ and $({\mathbf{H}}{\mathbf{1}})$-$({\mathbf{H}}{\mathbf{4}})$ hold. If $\kappa>2L_{B_2}^2$, then for any initial values $x\in H_1$, $y\in H_2$ and $T>0$, we have
\begin{align}
\mathbb{E} \left(\sup_{t\in[0,T]}\|X_{t}^{\vare}-\bar{X}_{t}\|_{H_1}^{2} \right)\leq C_T(1+\|x\|_{H_1}^2+\|y\|_{H_2}^2)\varepsilon^{1/3}\rightarrow0,~~~\text{as}~\varepsilon\rightarrow0,\label{2.2}
\end{align}
where $C_T$ is a constant only depending on $T$,  $\bar{X}_t$ is the solution of the following averaged equation
\begin{equation}\left\{\begin{array}{l}
\displaystyle d\bar{X}_t=\left[A_1(\bar{X}_t,\mathscr{L}_{\bar{X}_t})
+\bar{f}(\bar{X}_t,\mathscr{L}_{\bar{X}_t})\right]dt
+B_1(\bar{X}_t,\mathscr{L}_{\bar{X}_t})dW^1_t,\\
\bar{X}_{0}=x. \end{array}\right. \label{1.3}
\end{equation}
 Here the nonlinear coefficient $\bar{f}(x,\mu):=\int_{H_2}f(x,\mu,y)\nu^{x,\mu}(dy)$ is the average of $f$ with $\nu^{x,\mu}$ being the unique invariant distribution of the frozen equation below with respect to any fixed $x\in H_1$ and $\mu\in\mathscr{P}_2(H_1)$,
\begin{eqnarray*}
\left\{ \begin{aligned}
&dY_{t}=A_2(x,\mu,Y_t)dt
+B_2(x,\mu,Y_t)d \tilde{W}^{2}_{t},\\
&Y_{0}=y,
\end{aligned} \right.
\end{eqnarray*}
where $\tilde{{W}}_{t}^{{2}}$ is an $U_2$-valued cylindrical Wiener process defined on another
probability space $(\tilde{\Omega},\tilde{\mathscr{F}},\tilde{\mathbb{P}})$.
\end{theorem}

Throughout the paper, $C,C_{T}$ denote some positive constants which may change from line to line, and $C_{T}$ is used to stress that the constant only depends on $T$.
\section{Proof of Existence and uniqueness}
\setcounter{equation}{0}
 \setcounter{definition}{0}
In this section, we will use the technique of Galerkin type approximation to get the existence and uniqueness of strong solutions to system \eref{e1}.

Without loss of the generality, we assume $\varepsilon=1$ in system \eref{e1} and consider the following equation
\begin{eqnarray}\label{e2}
\left\{ \begin{aligned}
&dX^{1}_t=\left[A_1(X^{1}_t,\mathscr{L}_{X^{1}_t})
+f(X^{1}_t,\mathscr{L}_{X^{1}_t},Y^{2}_t)\right]dt
+B_1(X^{1}_t,\mathscr{L}_{X^{1}_t})dW^{1}_t,\\
&dY^{2}_t=A_2(X^{1}_t,\mathscr{L}_{X^{1}_t},Y^{2}_t)dt
+B_2(X^{1}_t,\mathscr{L}_{X^{1}_t},Y^{2}_t)d W^{2}_{t},
\\&X^{1}_0=x, Y^{2}_0=y.
\end{aligned} \right.
\end{eqnarray}
Choosing $\{e_1,e_2,\cdots\}\subset V_1$ as an orthonormal basis (ONB) on $H_1$ and $\{l_1,l_2,\cdots\}\subset V_2$ as an ONB on $H_2$. Define the following maps
$$\Pi^n_1:V_1^{*}\rightarrow H_1^n:=\text{span}\{e_1,e_2,\cdots\,e_n\},~\Pi^n_2:V_2^{*}\rightarrow H_2^n:=\text{span}\{l_1,l_2,\cdots\,l_n\},~n\geq1,$$
respectively by
$$\Pi^n_{1}x:=\sum\limits_{i=1}^{n}{}_{V_1^*}\langle x,e_i\rangle_{V}e_i,~x\in V_1^*,$$
and
$$\Pi^n_{2}y:=\sum\limits_{i=1}^{n}{}_{V_2^*}\langle y,l_i\rangle_{V_2}l_i,~y\in V_2^*.$$
It is easy to see that if we restrict $\Pi^n_1$ to $H_1$, denoted by $\Pi^n_1|_{H_1}$, then it is an orthogonal projection onto $H^n_1$ on $H_1$. Denote by $\{g_1,g_2,\cdots\}$ and $\{j_1,j_2,\cdots\}$ the ONBs on $U_1$ and $U_2$, respectively. Let
\begin{equation*}
W^{1,n}_t:=\widetilde{\Pi}^n_{1}W^1_t=\sum\limits_{i=1}^{n}\langle W^1_t,g_i\rangle_{U_1}g_i,~n\geq1,
\end{equation*}
and
\begin{equation*}
W^{2,n}_t:=\widetilde{\Pi}^n_{2}W^2_t=\sum\limits_{i=1}^{n}\langle W^2_t,j_i\rangle_{U_2}j_i,~n\geq1,
\end{equation*}
where $\widetilde{\Pi}_{1}^n$ is an orthonormal projection onto $U_1^n:=\text{span}\{g_1,g_2,\cdots,g_n\}$ on $U_1$, and analogously for $\widetilde{\Pi}_{2}^n$.

For any $n\geq1$, we consider the following finite dimensional equation
\begin{eqnarray}\label{ei}
\left\{ \begin{aligned}
&dX^{1,n}_t=\Pi^n_1\left[A_1(X^{1,n}_t,\mathscr{L}_{X^{1,n}_t})
+f(X^{1,n}_t,\mathscr{L}_{X^{1,n}_t},Y^{2,n}_t)\right]dt
+\Pi^n_1B_1(X^{1,n}_t,\mathscr{L}_{X^{1,n}_t})dW^{1,n}_t,\\
&dY^{2,n}_t=\Pi^n_2A_2(X^{1,n}_t,\mathscr{L}_{X^{1,n}_t},Y^{2,n}_t)dt
+\Pi^n_2B_2(X^{1,n}_t,\mathscr{L}_{X^{1,n}_t},Y^{2,n}_t)d W^{2,n}_{t},
\\&X^{1,n}_0=x^n, Y^{2,n}_0=y^n,
\end{aligned} \right.
\end{eqnarray}
here we denote $X^{1,n}:=\Pi^n_1X$, $Y^{2,n}:=\Pi^n_2Y$, $x^{n}:=\Pi^n_1x$ and $y^{n}:=\Pi^n_2y$.

We now introduce the following product spaces.
Let $\mathcal{H}:={H_1}\times H_2$ be the product Hilbert space. For any $\phi:=(\phi_1,\phi_2),\varphi:=(\varphi_1,\varphi_2)\in\mathcal{H}$, we denote the scalar product and the induced norm by
\begin{eqnarray*}
 \qquad\langle\phi,\varphi\rangle_{\mathcal{H}}=\langle\phi_1,\varphi_1\rangle_{H_1}
 +\langle\phi_2,\varphi_2\rangle_{H_2},~~
\|\phi\|_{\mathcal{H}}=\sqrt{\langle\phi,\phi\rangle_{\mathcal{H}}}=
\sqrt{\|\phi_1\|_{H_1}^2+\|\phi_2\|_{H_2}^2}.
\end{eqnarray*}
Similarly, we also define $\mathcal{U}:={U}_1\times U_2$ and $\mathcal{V}:={V}_1\times V_2$. Then $\mathcal{V}$ is a reflexive Banach space with the norm,
\begin{eqnarray*}\|\psi\|_{\mathcal{V}}=\sqrt{\langle\psi,\psi\rangle_{\mathcal{V}}}=\sqrt{\|\psi_1\|_{V_1}^2+\|\psi_2\|_{V_2}^2},
~~\text{for~any}~\psi=(\psi_1,\psi_2)\in \mathcal{V}.\end{eqnarray*}

We rewrite the systems \eref{e2} and \eref{ei} for $\Theta=(X^1,Y^2)$ and $\Theta^{n}=(X^{1,n},Y^{2,n})$, respectively, as
\begin{eqnarray}
d\Theta_t={A}(\Theta_t,\mathscr{L}_{\Theta_t}) dt+B(\Theta_t,\mathscr{L}_{\Theta_t})dW_t,\quad \Theta_0=(x,y),\label{E4.1}
\end{eqnarray}
\begin{eqnarray}
d\Theta^{n}_t=\Pi^{n}{A}(\Theta^{n}_t,\mathscr{L}_{\Theta^{n}_t}) dt+\Pi^{n}B(\Theta^{n}_t,\mathscr{L}_{\Theta^{n}_t})dW^{n}_t,\quad \Theta^{n}_0=(x^{n},y^{n}),\label{E4.2}
\end{eqnarray}
where $\mathscr{L}_{\Theta}\in C([0,T];\mathscr{P}_2(\mathcal{H}))$ with its marginal distribution $\mathscr{L}_{X^{1}}\in C([0,T];\mathscr{P}_2(H_1))$, analogously for $\mathscr{L}_{\Theta^{n}}$, $\Pi^{n}:=diag\Big(\Pi^{n}_1,\Pi^{n}_2\Big)$ and
\begin{eqnarray*}
&&{A}(\Theta_t,\mathscr{L}_{\Theta_t}):=\left(A_1(X^{1}_t,\mathscr{L}_{X^{1}_t})
+f(X^{1}_t,\mathscr{L}_{X^{1}_t},Y^{2}_t),
A_2(X^{1}_t,\mathscr{L}_{X^{1}_t},Y^{2}_t)\right),
\\&&{A}(\Theta^{n}_t,\mathscr{L}_{\Theta^{n}_t}):=\left(A_1(X^{1,n}_t,\mathscr{L}_{X^{1,n}_t})
+f(X^{1,n}_t,\mathscr{L}_{X^{1,n}_t},Y^{2,n}_t),
A_2(X^{1,n}_t,\mathscr{L}_{X^{1,n}_t},Y^{2,n}_t)\right),
\\&&B(\Theta_t,\mathscr{L}_{\Theta_t}):=diag\left( B_1(X^{1}_t,\mathscr{L}_{X^{1}_t}), B_2(X^{1}_t,\mathscr{L}_{X^{1}_t},Y^{2}_t)\right),
\\&&B(\Theta^{n}_t,\mathscr{L}_{\Theta^{n}_t}):=diag\left( B_1(X^{1,n}_t,\mathscr{L}_{X^{1,n}_t}), B_2(X^{1,n}_t,\mathscr{L}_{X^{1,n}_t},Y^{2,n}_t)\right),
\end{eqnarray*}
and $W_t:=(W_t^{1},W_t^{2})$, $W_t^{n}:=(W_t^{1,n},W_t^{2,n})$.
Let $L_2(\mathcal{U},\mathcal{H})$ denotes the space of Hilbert-Schmidt operators from $\mathcal{U}$ to $\mathcal{H}$, with the norm:
$$
\|S\|_{L_2(\mathcal{U},\mathcal{H})}:=\sqrt{\|S_1\|^2_{L_2({U_1},{H_1})}
+\|S_2\|^2_{L_2({U_2},{H_2})}},\quad S=(S_1, S_2),
$$
where $S_i\in L_2({U_i},{H_i})$, $i=1,2$.
Let $\mathcal{V}^{*}$
be the dual space of $\mathcal{V}$, it is obvious that $\mathcal{V}^{*}=V_1^*\times V_2^*$, and we consider the following  Gelfand triple $$\mathcal{V}\subset\mathcal{H}\cong\mathcal{H}^{*}\subset\mathcal{V}^{*}.$$
It is easy to see that the following mappings
$${A}:\mathcal{V}\times\mathscr{P}_2(\mathcal{H})\rightarrow\mathcal{V}^{*},
~~B:\mathcal{V}\times\mathscr{P}_2(\mathcal{H})\rightarrow\mathcal{V}^{*}$$
are well defined.

To complete the proof, we first verify the new coefficients in equation (\ref{E4.1}) satisfy the monotonicity condition similar to \eref{m}. Indeed, for any $w_1=(u_1,v_1),w_2=(u_2,v_2)\in\mathcal{V}$, and $\vartheta_1=(\mu_1,\nu_1),\vartheta_2=(\mu_2,\nu_2)\in\mathscr{P}_2(\mathcal{H})$, by conditions $({\mathbf{A}}{\mathbf{2}})$ and $({\mathbf{H}}{\mathbf{2}})$, we have
\begin{eqnarray}\label{mon}
&&{_{\mathcal{V}^{*}}}\langle A(w_1,\vartheta_1)-A(w_2,\vartheta_2),w_1-w_2\rangle_{\mathcal{V}}
\nonumber\\=\!\!\!\!\!\!\!\!&&{_{{V_1}^{*}}}\langle A_1(u_1,\mu_1)-A_1(u_2,\mu_2),u_1-u_2\rangle_{{V_1}}
 +\langle f(u_1,\mu_1,v_1)-f(u_2,\mu_2,v_2),u_1-u_2\rangle_{H_1}
\nonumber\\\!\!\!\!\!\!\!\!&&+{_{{V_2}^{*}}}\langle A_2(u_1,\mu_1,v_1)-A_2(u_2,\mu_2,v_2),v_1-v_2\rangle_{{V_2}}
\nonumber\\\leq\!\!\!\!\!\!\!\!&&C(\|u_1-u_2\|_{H_1}^2+\|v_1-v_2\|_{H_2}^2+\mathbb{W}_{2,H_1}(\mu_1,\mu_2)^2)
\nonumber\\\leq\!\!\!\!\!\!\!\!&&C(\|w_1-w_2\|_{\mathcal{H}}^2+\mathbb{W}_{2,\mathcal{H}}(\vartheta_1,\vartheta_2)^2)
\end{eqnarray}
and
\begin{eqnarray}\label{Lip}
&&\|B(w_1,\vartheta_1)-B(w_2,\vartheta_2)\|^2_{L_2(\mathcal{U},\mathcal{H})}\nonumber\\
=\!\!\!\!\!\!\!\!&&\|B_1(u_1,\mu_1)-B_1(u_2,\mu_2)\|^2_{{L_2({U_1},{H_1})}}
+\|B_2(u_1,\mu_1,v_1)-B_2(u_2,\mu_2,v_2)\|^2_{{L_2({U_2},{H_2})}}\nonumber\\
\leq\!\!\!\!\!\!\!\!&&C(\|w_1-w_2\|_{\mathcal{H}}^2+\mathbb{W}_{2,\mathcal{H}}(\vartheta_1,\vartheta_2)^2).
\end{eqnarray}

Following from \cite[Lemma 2.2]{RTW} or \cite[Theorem 3.3]{HRW}, by \eref{mon}, \eref{Lip}, $({\mathbf{A}}{\mathbf{1}})$, $({\mathbf{H}}{\mathbf{1}})$, $({\mathbf{A}}{\mathbf{4}})$ and $({\mathbf{H}}{\mathbf{4}})$, system (\ref{ei}) has a unique continuous solution $(X^{1,n},Y^{2,n})$. We now define the following spaces equipped with the associated norms
\begin{eqnarray*}
\!\!\!\!\!\!\!\!&&J_i:=L^2([0,T]\times\Omega,dt\times\mathbb{P};L_2(U_i,H_i)),~i=1,2,
\nonumber\\
\!\!\!\!\!\!\!\!&&K_1:=L^\alpha([0,T]\times\Omega,dt\times\mathbb{P};V_1),~K_2:=L^\beta([0,T]\times\Omega,dt\times\mathbb{P};V_2),
\nonumber\\
\!\!\!\!\!\!\!\!&&K_1^*:=L^{\frac{\alpha}{\alpha-1}}([0,T]\times\Omega,dt\times\mathbb{P};V_1^*),~K_2^*:=L^{\frac{\beta}{\beta-1}}([0,T]\times\Omega,dt\times\mathbb{P};V_2^*).
\end{eqnarray*}

In order to prove the existence of solutions, we first give the following apriori estimates.
\begin{lemma}\label{l4}
Suppose that $({\mathbf{A}}{\mathbf{3}})$ and $({\mathbf{H}}{\mathbf{3}})$ hold. Then there exists a constant $C_T>0$, which is independent of $n$, such that for all $n\geq1$
$$\mathbb{E}\big[\sup_{t\in[0,T]}\|X^{1,n}_t\|_{H_1}^2\big]+\mathbb{E}\big[\sup_{t\in[0,T]}\|Y^{2,n}_t\|_{H_2}^2\big]+\|X^{1,n}\|_{K_1}+\|Y^{2,n}\|_{K_2}\leq C_T(1+\|x\|_{H_1}^2+\|y\|_{H_2}^2).$$
\end{lemma}
\begin{proof}
By It\^{o}'s formula for finite dimensional case and $({\mathbf{H}}{\mathbf{3}})$, we have
\begin{eqnarray*}
d\|Y_t^{2,n}\|_{H_2}^2=\!\!\!\!\!\!\!\!&&\Big[2{}_{V_2^*}\big\langle \Pi^n_2A_2(X_t^{1,n},\mathscr{L}_{X_t^{1,n}},Y_t^{2,n}),Y_t^{2,n}\big\rangle_{V_2}+\|\Pi^n_2B_2(X_t^{1,n},\mathscr{L}_{X_t^{1,n}},Y_t^{2,n})\widetilde{\Pi}^n_{2}\|_{L_2(U_2,H_2)}^2\Big]dt
\nonumber\\
\!\!\!\!\!\!\!\!&&+2\big\langle\Pi^n_2B_2(X_t^{1,n},\mathscr{L}_{X_t^{1,n}},Y_t^{2,n})dW^{2,n}_t,Y_t^{2,n}\big\rangle_{H_2}
\nonumber\\
\leq\!\!\!\!\!\!\!\!&&\Big[-\eta\|Y_t^{2,n}\|_{V_2}^{\beta}+c_2\big(1+\|X_t^{1,n}\|_{H_1}^2+\mathscr{L}_{X_t^{1,n}}(\|\cdot\|_{H_1}^2)+\|Y_t^{2,n}\|_{H_2}^2\big)\Big]dt+dM^n_t,
\end{eqnarray*}
where we denote $dM^n_t:=2\big\langle\Pi^n_2B_2(X_t^{1,n},\mathscr{L}_{X_t^{1,n}},Y_t^{2,n})dW^{2,n}_t,Y_t^{2,n}\big\rangle_{H_2}$.

We set the following stopping time
$$\tau_R^{n}:=\inf\big\{t\in[0,T]:\|X^{1,n}_t\|_{H_1}+\|Y^{2,n}_t\|_{H_2}>R\big\},~R>0.$$
Then, using Burkholder-Davis-Gundy's inequality, we have
\begin{eqnarray*}
\!\!\!\!\!\!\!\!&&\mathbb{E}\big[\sup_{t\in[0,T\wedge\tau_R^{n}]}\|Y_t^{2,n}\|_{H_2}^2\big]+\eta\mathbb{E}\int_0^{T\wedge\tau_R^{n}}\|Y_t^{2,n}\|_{V_2}^{\beta}dt
\nonumber\\
=\!\!\!\!\!\!\!\!&&\|y^n\|_{H_2}^2+CT+C\mathbb{E}\int_0^{T\wedge\tau_R^{n}}\|Y_t^{2,n}\|_{H_2}^2dt
\nonumber\\
\!\!\!\!\!\!\!\!&&+
C\mathbb{E}\int_0^{T\wedge\tau_R^{n}}\big(\|X_t^{1,n}\|_{H_1}^2+\mathscr{L}_{X_t^{1,n}}(\|\cdot\|_{H_1}^2)\big)dt
+\mathbb{E}\big[\sup_{t\in[0,T\wedge\tau_R^{n}]}|M^n_t|\big]
\nonumber\\
\leq\!\!\!\!\!\!\!\!&&\|y\|_{H_2}^2+\frac{1}{2}\mathbb{E}\big[\sup_{t\in[0,T\wedge\tau_R^{n}]}\|Y_t^{2,n}\|_{H_2}^2\big]+CT+C\mathbb{E}\int_0^{T\wedge\tau_R^{n}}\|Y_t^{2,n}\|_{H_2}^2dt
\nonumber\\
\!\!\!\!\!\!\!\!&&+
C\mathbb{E}\int_0^{T\wedge\tau_R^{n}}\big(\|X_t^{1,n}\|_{H_1}^2+\mathscr{L}_{X_t^{1,n}}(\|\cdot\|_{H_1}^2)\big)dt,
\end{eqnarray*}
which implies
\begin{eqnarray}\label{1}
\!\!\!\!\!\!\!\!&&\mathbb{E}\big[\sup_{t\in[0,T\wedge\tau_R^{n}]}\|Y_t^{2,n}\|_{H_2}^2\big]+2\eta\mathbb{E}\int_0^{T\wedge\tau_R^{n}}\|Y_t^{2,n}\|_{V_2}^{\beta}dt
\nonumber\\
\leq\!\!\!\!\!\!\!\!&&C\|y\|_{H_2}^2+C_T+C\int_0^{T}\mathbb{E}\sup_{s\in[0,t\wedge\tau_R^{n}]}\|Y_t^{2,n}\|_{H_2}^2dt
\nonumber\\
\!\!\!\!\!\!\!\!&&
+C\mathbb{E}\int_0^{T\wedge\tau_R^{n}}\big(\|X_t^{1,n}\|_{H_1}^2+\mathscr{L}_{X_t^{1,n}}(\|\cdot\|_{H_1}^2)\big)dt.~~
\end{eqnarray}
Applying Gronwall's inequality, we obtain
\begin{eqnarray}\label{2}
\!\!\!\!\!\!\!\!&&\mathbb{E}\big[\sup_{t\in[0,T\wedge\tau_R^{n}]}\|Y_t^{2,n}\|_{H_2}^2\big]+2\eta\mathbb{E}\int_0^{T\wedge\tau_R^{n}}\|Y_t^{2,n}\|_{V_2}^{\beta}dt
\nonumber\\
\leq\!\!\!\!\!\!\!\!&&C_T\|y\|_{H_2}^2+C_T
+C_T\mathbb{E}\int_0^{T\wedge\tau_R^{n}}\big(\|X_t^{1,n}\|_{H_1}^2+\mathscr{L}_{X_t^{1,n}}(\|\cdot\|_{H_1}^2)\big)dt.
\end{eqnarray}
Similarly, applying It\^{o}'s formula to $\|X_t^{1,n}\|_{H_1}^2$ and using $({\mathbf{A}}{\mathbf{3}})$, we have
\begin{eqnarray*}
d\|X_t^{1,n}\|_{H_1}^2=\!\!\!\!\!\!\!\!&&\Big[2_{V_1^*}\big\langle \Pi^n_1A_1(X_t^{1,n},\mathscr{L}_{X_t^{1,n}}),X_t^{1,n}\big\rangle_{V_1}+2\big\langle \Pi^n_1f(X_t^{1,n},\mathscr{L}_{X_t^{1,n}},Y_t^{2,n}),X_t^{1,n}\big\rangle_{H_1}
\nonumber\\
\!\!\!\!\!\!\!\!&&
+\|\Pi^n_1B_1(X_t^{1,n},\mathscr{L}_{X_t^{1,n}})\widetilde{\Pi}^n_{1}\|_{L_2(U_1,H_1)}^2\Big]dt+dN^{n}_t
\nonumber\\
\leq\!\!\!\!\!\!\!\!&&\Big[-\theta\|X_t^{1,n}\|_{V_1}^{\alpha}+C\big(1+\|X_t^{1,n}\|_{H_1}^2+\mathscr{L}_{X_t^{1,n}}(\|\cdot\|_{H_1}^2)\big)+C\|Y_t^{2,n}\|_{H_2}^2\Big]dt+dN^{n}_t,
\end{eqnarray*}
here we denote $dN^{n}_t:=2\big\langle\Pi^n_1B_1(X_t^{1,n},\mathscr{L}_{X_t^{1,n}})dW^{1,n}_t,X_t^{1,n}\big\rangle_{H_1}$.

By Burkholder-Davis-Gundy's inequality we infer that
\begin{eqnarray*}
\!\!\!\!\!\!\!\!&&\mathbb{E}\big[\sup_{t\in[0,T\wedge\tau_R^{n}]}\|X_t^{1,n}\|_{H_1}^2\big]+\theta\mathbb{E}\int_0^{T\wedge\tau_R^{n}}\|X_t^{1,n}\|_{V_1}^{\alpha}dt
\nonumber\\
\leq\!\!\!\!\!\!\!\!&&\|x\|_{H_1}^2+C_T+C\mathbb{E}\int_0^{T\wedge\tau_R^{n}}\big(\|X_t^{1,n}\|_{H_1}^2+\mathscr{L}_{X_t^{1,n}}(\|\cdot\|_{H_1}^2)\big)dt
\nonumber\\
\!\!\!\!\!\!\!\!&&
+C\mathbb{E}\int_0^{T\wedge\tau_R^{n}}\|Y_t^{2,n}\|_{H_2}^2dt+\mathbb{E}\big[\sup_{t\in[0,T\wedge\tau_R^{n}]}|N^{n}_t|\big]
\nonumber\\
\leq\!\!\!\!\!\!\!\!&&C_T(1+\|x\|_{H_1}^2+\|y\|_{H_2}^2)+C_T\int_0^T\mathbb{E}\|X_t^{1,n}\|_{H_1}^2dt+\frac{1}{2}\mathbb{E}\big[\sup_{t\in[0,T\wedge\tau_R^{n}]}\|X_t^{1,n}\|_{H_1}^2\big],
\end{eqnarray*}
where we used \eref{2} in the last step.

Rearranging the above inequality and taking $R\to\infty$, then by the monotone convergence theorem we have
\begin{eqnarray*}
\!\!\!\!\!\!\!\!&&\mathbb{E}\big[\sup_{t\in[0,T]}\|X_t^{1,n}\|_{H_1}^2\big]+2\theta\mathbb{E}\int_0^{T}\|X_t^{1,n}\|_{V_1}^{\alpha}dt
\nonumber\\
\leq\!\!\!\!\!\!\!\!&&C_T(1+\|x\|_{H_1}^2+\|y\|_{H_2}^2)+C_T\int_0^T\mathbb{E}\|X_t^{1,n}\|_{H_1}^2dt.
\end{eqnarray*}
Thus applying Gronwall's lemma gives that
\begin{eqnarray}\label{3}
\mathbb{E}\big[\sup_{t\in[0,T]}\|X_t^{1,n}\|_{H_1}^2\big]+2\theta\mathbb{E}\int_0^{T}\|X_t^{1,n}\|_{V_1}^{\alpha}dt
\leq C_T(1+\|x\|_{H_1}^2+\|y\|_{H_2}^2).
\end{eqnarray}
Recalling \eref{1} and following the same procedure as \eref{3}, it is easy to get that
\begin{eqnarray*}
\mathbb{E}\big[\sup_{t\in[0,T]}\|Y_t^{2,n}\|_{H_2}^2\big]+2\eta\mathbb{E}\int_0^{T}\|Y_t^{2,n}\|_{V_2}^{\beta}dt
\leq C_T(1+\|x\|_{H_1}^2+\|y\|_{H_2}^2),
\end{eqnarray*}
which completes the proof. \hspace{\fill}$\Box$
\end{proof}

Combining $({\mathbf{A}}{\mathbf{3}})$, $({\mathbf{A}}{\mathbf{4}})$, $({\mathbf{H}}{\mathbf{3}})$, $({\mathbf{H}}{\mathbf{4}})$ with Lemma \ref{l4}, it is easy to obtain the following estimates.
\begin{lemma}\label{l5}
Suppose $({\mathbf{A}}{\mathbf{3}})$, $({\mathbf{A}}{\mathbf{4}})$, $({\mathbf{H}}{\mathbf{3}})$ and $({\mathbf{H}}{\mathbf{4}})$. There exists a constant $C_T>0$ which is independent of $n$ such that
\begin{eqnarray*}
\!\!\!\!\!\!\!\!&&\|A_1(X_{\cdot}^{1,n},\mathscr{L}_{X_{\cdot}^{1,n}})\|_{K_1^*}+\|f(X_{\cdot}^{1,n},\mathscr{L}_{X_{\cdot}^{1,n}},Y_{\cdot}^{2,n})\|_{L^2([0,T]\times\Omega;H_1)}+\|B_1(X_{\cdot}^{1,n},\mathscr{L}_{X_{\cdot}^{1,n}})\|_{J_1}  \nonumber\\
\!\!\!\!\!\!\!\!&&~~~~~\leq C_T(1+\|x\|_{H_1}^2+\|y\|_{H_2}^2),
\end{eqnarray*}
and
\begin{eqnarray*}
\!\!\!\!\!\!\!\!&&\|A_2(X_{\cdot}^{1,n},\mathscr{L}_{X_{\cdot}^{1,n}},Y_{\cdot}^{2,n})\|_{K_2^*}+\|B_1(X_{\cdot}^{1,n},\mathscr{L}_{X_{\cdot}^{1,n}},Y_{\cdot}^{2,n})\|_{J_2}  \leq C_T(1+\|x\|_{H_1}^2+\|y\|_{H_2}^2),
\end{eqnarray*}
for all $n\geq1$.
\end{lemma}

\noindent
\textbf{Proof of Theorem \ref{Th1}:} Due to the reflexivity of $J_i,K_i,K_i^*$ and $L^2([0,T]\times\Omega,dt\times\mathbb{P};H_i)$, $i=1,2$, there exist common subsequences $n_k$ such that for $k\rightarrow\infty$,
\begin{eqnarray*}
\!\!\!\!\!\!\!\!&&(i)~X^{1,n_k}\rightarrow\bar{X}~\text{weakly in~} K_1~\text{and weakly in}~L^2([0,T]\times\Omega,dt\times\mathbb{P};H_1),
\nonumber\\
\!\!\!\!\!\!\!\!&&(ii)~Y^{2,n_k}\rightarrow\bar{Y}~\text{weakly in~} K_2~\text{and weakly in}~L^2([0,T]\times\Omega,dt\times\mathbb{P};H_2),
\nonumber\\
\!\!\!\!\!\!\!\!&&(iii)~A_1(X^{1,n_k}_{\cdot},\mathscr{L}_{X^{1,n_k}_{\cdot}})\rightarrow \bar{F}^1~\text{weakly in~} K_1^*,
\nonumber\\
\!\!\!\!\!\!\!\!&&(iv)~A_2(X^{1,n_k}_{\cdot},\mathscr{L}_{X^{1,n_k}_{\cdot}},Y^{2,n_k}_{\cdot})\rightarrow \bar{F}^2~\text{weakly in~} K_2^*,
\nonumber\\
\!\!\!\!\!\!\!\!&&(v)~f(X^{1,n_k}_{\cdot},\mathscr{L}_{X^{1,n_k}_{\cdot}},Y^{2,n_k}_{\cdot})\rightarrow \hat{f}~\text{weakly in~} L^2([0,T]\times\Omega,dt\times\mathbb{P};H_1),
\nonumber\\
\!\!\!\!\!\!\!\!&&(vi)~B_1(X^{1,n_k}_{\cdot},\mathscr{L}_{X^{1,n_k}_{\cdot}})\rightarrow \bar{Z}^1~\text{weakly in~} J_1,
\nonumber\\
\!\!\!\!\!\!\!\!&&(vii)~B_2(X^{1,n_k}_{\cdot},\mathscr{L}_{X^{1,n_k}_{\cdot}},Y^{2,n_k}_{\cdot})\rightarrow \bar{Z}^2~\text{weakly in~} J_2.
\end{eqnarray*}

Due to $\alpha\geq2$, it is obvious that
$$f(X^{1,n_k}_{\cdot},\mathscr{L}_{X^{1,n_k}_{\cdot}},Y^{2,n_k}_{\cdot})\rightarrow \hat{f}~\text{weakly in~}K_1^*.$$
Since the bounded linear operator between two Banach space is weakly continuous, it leads to $\int_0^{\cdot}\Pi^{n_k}_1B_1(X^{1,n_k}_{s},\mathscr{L}_{X^{1,n_k}_{s}})dW^{1,n_k}_s\to\int_0^{\cdot}\bar{Z}^1_s\mathrm{d}W^1_s$ weakly in $\mathcal{M}_T^2(H_1)$ (the space of all continuous square integrable martingales from $[0,T]\times\Omega$ to $H_1$), and analogously for $\int_0^{\cdot}\Pi^{n_k}_2B_2(X^{1,n_k}_{s},\mathscr{L}_{X^{1,n_k}_{s}},Y^{2,n_k}_s)dW^{2,n_k}_s$. In addition, the approximants are progressively measurable, it follows that all of the above limits are progressively measurable.

Note that  $V_i$, $i=1,2$, are separable, by the definition of $(X^{1,n_k},Y^{2,n_k})$ that for any $\bar{e}\in V_1$ and $\tilde{e}\in V_2$, $dt\times\mathbb{P}$-a.e.
\begin{eqnarray*}
\left\{ \begin{aligned}
&_{V_1^*}\langle\bar{X}_t,\bar{e}\rangle_{V_1}={}_{V_1^*}\langle x,\bar{e}\rangle_{V_1}+\int_0^t{}_{V_1^*}\langle\bar{F}^1_s,\bar{e}\rangle_{V_1}ds
+\int_0^t\langle \hat{f}_s,\bar{e}\rangle_{H_1}ds
+\int_0^t\langle \bar{Z}^1_sdW^1_s,\bar{e}\rangle_{H_1},\\
&_{V_2^*}\langle\bar{Y}_t,\tilde{e}\rangle_{V_2}={}_{V_2^*}\langle y,\tilde{e}\rangle_{V_2}+\int_0^t{}_{V_2^*}\langle\bar{F}^2_s,\tilde{e}\rangle_{V_2}ds
+\int_0^t\langle \bar{Z}^2_sdW^2_s,\tilde{e}\rangle_{H_2},
\end{aligned} \right.
\end{eqnarray*}
Let us define
\begin{eqnarray*}
\left\{ \begin{aligned}
&X^1_t=x+\int_0^t\bar{F}^1_sds+\int_0^t \hat{f}_sds+\int_0^t\bar{Z}^1_sdW^1_s,~t\in[0,T],\\
&Y^2_t=y+\int_0^t\bar{F}^2_sds+\int_0^t\bar{Z}^2_sdW^2_s,~t\in[0,T],
\end{aligned} \right.
\end{eqnarray*}
then it gives that $X^1=\bar{X}$, $Y^2=\bar{Y}$, $dt\times\mathbb{P}$-a.e. According to Lemma \ref{l4} and \cite[Theorem 4.2.5]{LR1} that $(X^1,Y^2)$ is a continuous $H_1\times H_2$-valued $(\mathscr{F}_t)$-adapted process.
Thus, it suffices to prove that $dt\times\mathbb{P}$-a.e.
\begin{equation}\label{lim1}
\bar{F}^1=A_1(X^{1}_{\cdot},\mathscr{L}_{X^{1}_{\cdot}}),~\hat{f}=f(X^{1}_{\cdot},\mathscr{L}_{X^{1}_{\cdot}},Y^{2}_{\cdot}),~\bar{F}^2=A_2(X^{1}_{\cdot},\mathscr{L}_{X^{1}_{\cdot}},Y^{2}_{\cdot}),
\end{equation}
and
\begin{equation}\label{lim2}
\bar{Z}^1=B_1(X^{1}_{\cdot},\mathscr{L}_{X^{1}_{\cdot}}),~\bar{Z}^2=B_2(X^{1}_{\cdot},\mathscr{L}_{X^{1}_{\cdot}},Y^{2}_{\cdot}),
\end{equation}
which implies the existence of solutions to system \eref{e2}.

From $(i)$-$(vii)$ above, it is easy to get the corresponding convergence for $\Theta^{n_k}$, ${A}(\Theta^{n_k},\mathscr{L}_{\Theta^{n_k}})$ and $B(\Theta^{n_k},\mathscr{L}_{\Theta^{n_k}})$, respectively. For instance,
$${A}(\Theta^{n_k},\mathscr{L}_{\Theta^{n_k}})\to\bar{F}:=(\bar{F}^1+\hat{f},\bar{F}^2)~\text{weakly in}~K_1^*\times K_2^*,~\text{as}~k\to\infty.$$
By \eref{lim1} and \eref{lim2}, it suffices to show that $dt\times\mathbb{P}$-a.e.
\begin{equation}\label{lim3a}
{A}(\Theta_{\cdot},\mathscr{L}_{\Theta_{\cdot}})=\bar{F}~\text{and}~B(\Theta_{\cdot},\mathscr{L}_{\Theta_{\cdot}})=\bar{Z}:=diag(\bar{Z}^1,\bar{Z}^2).
\end{equation}
Combining \eref{mon}, \eref{Lip} with $({\mathbf{A}}{\mathbf{1}})$ and $({\mathbf{H}}{\mathbf{1}})$, \eref{lim3a} follows from the monotonicity arguments, we include the details here for completeness.

Take any non-negative $\psi\in L^\infty([0,T],dt;\mathbb{R})$, by H\"{o}lder's inequality, we have
\begin{eqnarray*}
\!\!\!\!\!\!\!\!&&\mathbb{E}\left[\int_0^T\psi_t\|\Theta_t\|_{\mathcal{H}}^2dt\right]=\lim_{k\to\infty}\mathbb{E}\left[\int_0^T\langle\psi_t\Theta_t,\Theta^{n_k}_t\rangle_{\mathcal{H}}dt\right]
\nonumber\\
\!\!\!\!\!\!\!\!&&~~\leq\left[\mathbb{E}\left(\int_0^T\psi_t\|\Theta_t\|_{\mathcal{H}}^2dt\right)\right]^{1/2}\liminf_{k\to\infty}\left[\mathbb{E}\left(\int_0^T\psi_t\|\Theta^{n_k}_t\|_{\mathcal{H}}^2dt\right)\right]^{1/2},
\end{eqnarray*}
which implies the following lower semi-continuity
\begin{eqnarray}\label{4}
\mathbb{E}\left[\int_0^T\psi_t\|\Theta_t\|_{\mathcal{H}}^2dt\right]\leq\liminf_{k\to\infty}\mathbb{E}\left[\int_0^T\psi_t\|\Theta^{n_k}_t\|_{\mathcal{H}}^2dt\right].
\end{eqnarray}

For any $\phi:=(\phi_1,\phi_2)\in K_1\times K_2$ and $\lambda\geq 0$, It\^{o}'s formula yields that
\begin{eqnarray}\label{5}
\!\!\!\!\!\!\!\!&&\mathbb{E}\Big[e^{-\lambda t}\|\Theta^{n_k}_t\|_{\mathcal{H}}^2\Big]-\|\Theta^{n_k}_0\|_{\mathcal{H}}^2
\nonumber\\
\leq\!\!\!\!\!\!\!\!&&\mathbb{E}\Big[\int_0^te^{-\lambda s}\Big(2{}_{\mathcal{V^*}}\langle {A}(\Theta^{n_k}_{s},\mathscr{L}_{\Theta^{n_k}_{s}}), \Theta^{n_k}_{s}\rangle_{\mathcal{V}}+\|B(\Theta^{n_k}_{s},\mathscr{L}_{\Theta^{n_k}_{s}})\|_{L_2(\mathcal{U},\mathcal{H})}^2-\lambda\|\Theta^{n_k}_{s}\|_{\mathcal{H}}^2\Big)ds\Big]
\nonumber\\
=\!\!\!\!\!\!\!\!&&\mathbb{E}\Big[\int_0^te^{-\lambda s}\Big(2{}_{\mathcal{V^*}}\langle {A}(\Theta^{n_k}_{s},\mathscr{L}_{\Theta^{n_k}_{s}})-{A}(\phi_{s},\mathscr{L}_{\phi_{s}}), \Theta^{n_k}_{s}-\phi_s\rangle_{\mathcal{V}}
\nonumber\\
\!\!\!\!\!\!\!\!&&~~~~~~~~~~~~~~~+\|B(\Theta^{n_k}_{s},\mathscr{L}_{\Theta^{n_k}_{s}})-B(\phi_s,\mathscr{L}_{\phi_s})\|_{L_2(\mathcal{U},\mathcal{H})}^2-\lambda\|\Theta^{n_k}_{s}-\phi_s\|_{\mathcal{H}}^2\Big)ds\Big]
\nonumber\\
\!\!\!\!\!\!\!\!&&+\mathbb{E}\Big[\int_0^te^{-\lambda s}\Big(2{}_{\mathcal{V^*}}\langle{A}(\phi_s,\mathscr{L}_{\phi_s}),\Theta^{n_k}_{s}\rangle_{\mathcal{V}}+2{}_{\mathcal{V^*}}\langle A(\Theta^{n_k}_{s},\mathscr{L}_{\Theta^{n_k}_{s}})-{A}(\phi_s,\mathscr{L}_{\phi_s}),\phi_s\rangle_{\mathcal{V}}
\nonumber\\
\!\!\!\!\!\!\!\!&&~~~~~~~~~~~~~~~+2\langle B(\Theta^{n_k}_{s},\mathscr{L}_{\Theta^{n_k}_{s}}),B(\phi_s,\mathscr{L}_{\phi_s})\rangle_{L_2(\mathcal{U},\mathcal{H})}-\|B(\phi_s,\mathscr{L}_{\phi_s})\|_{L_2(\mathcal{U},\mathcal{H})}^2
\nonumber\\
\!\!\!\!\!\!\!\!&&~~~~~~~~~~~~~~~-2\lambda\langle\Theta^{n_k}_{s},\phi_s\rangle_{\mathcal{H}}+\lambda\|\phi_s\|_{\mathcal{H}}^2\Big)ds\Big].
\end{eqnarray}
By \eref{mon} and \eref{Lip}, it is easy to find a constant $c>0$ and take $\lambda=c$ such that
\begin{eqnarray}\label{6}
\!\!\!\!\!\!\!\!&&\mathbb{E}\Big[\int_0^te^{-\lambda s}\Big(2{}_{\mathcal{V^*}}\langle {A}(\Theta^{n_k}_{s},\mathscr{L}_{\Theta^{n_k}_{s}})-{A}(\phi_{s},\mathscr{L}_{\phi_{s}}), \Theta^{n_k}_{s}-\phi_s\rangle_{\mathcal{V}}
\nonumber\\
\!\!\!\!\!\!\!\!&&~~~~~~~~~~~~~~~+\|B(\Theta^{n_k}_{s},\mathscr{L}_{\Theta^{n_k}_{s}})-B(\phi_s,\mathscr{L}_{\phi_s})\|_{L_2(\mathcal{U},\mathcal{H})}^2-\lambda\|\Theta^{n_k}_{s}-\phi_s\|_{\mathcal{H}}^2\Big)ds\Big]
\nonumber\\
\leq\!\!\!\!\!\!\!\!&&\mathbb{E}\Big\{\int_0^te^{-\lambda s}\Big[c\Big(\|\Theta^{n_k}_{s}-\phi_{s}\|_{\mathcal{H}}^2+\mathbb{W}_{2,\mathcal{H}}(\mathscr{L}_{\Theta^{n_k}_{s}},\mathscr{L}_{\phi_s})^2\Big)-\lambda\|\Theta^{n_k}_{s}-\phi_s\|_{\mathcal{H}}^2\Big]ds\Big\}
\nonumber\\
=\!\!\!\!\!\!\!\!&&0.
\end{eqnarray}
Inserting (\ref{6}) into (\ref{5}), by the lower semi-continuity (\ref{4}) we get
\begin{eqnarray}\label{7}
\!\!\!\!\!\!\!\!&&\mathbb{E}\left[\int_0^T\psi_t\Big(e^{-\lambda t}\|\Theta_t\|_{\mathcal{H}}^2-\|\Theta_0\|_{\mathcal{H}}^2\Big)dt\right]
\nonumber\\
\leq\!\!\!\!\!\!\!\!&&\liminf_{k\to\infty}\mathbb{E}\left[\int_0^T\psi_t\Big(e^{-\lambda t}\|\Theta^{n_k}_t\|_{\mathcal{H}}^2-\|\Theta_0\|_{\mathcal{H}}^2\Big)dt\right]
\nonumber\\
\leq\!\!\!\!\!\!\!\!&&\mathbb{E}\Big\{\int_0^T\psi_t\Big[\int_0^te^{-\lambda s}\Big(2{}_{\mathcal{V^*}}\langle{A}(\phi_s,\mathscr{L}_{\phi_s}),\Theta_{s}\rangle_{\mathcal{V}}+2{}_{\mathcal{V^*}}\langle \bar{F}_s-{A}(\phi_s,\mathscr{L}_{\phi_s}),\phi_s\rangle_{\mathcal{V}}
\nonumber\\
\!\!\!\!\!\!\!\!&&~~~~~~~~~~~~~~~+2\langle \bar{Z}_s,B(\phi_s,\mathscr{L}_{\phi_s})\rangle_{L_2(\mathcal{U},\mathcal{H})}-\|B(\phi_s,\mathscr{L}_{\phi_s})\|_{L_2(\mathcal{U},\mathcal{H})}^2
\nonumber\\
\!\!\!\!\!\!\!\!&&~~~~~~~~~~~~~~~-2\lambda\langle\Theta_{s},\phi_s\rangle_{\mathcal{H}}+\lambda\|\phi_s\|_{\mathcal{H}}^2\Big)ds\Big]dt\Big\}.
\end{eqnarray}
Applying It\^{o}'s formula and the product rule,
\begin{eqnarray}\label{16}
\!\!\!\!\!\!\!\!&&\mathbb{E}\left[e^{-\lambda t}\|\Theta_t\|_{\mathcal{H}}^2\right]-\|\Theta_0\|_{\mathcal{H}}^2
\nonumber\\
=\!\!\!\!\!\!\!\!&&\mathbb{E}\left[\int_0^te^{-\lambda s}\left(2{}_{\mathcal{V^*}}\langle {A}(\Theta_{s},\mathscr{L}_{\Theta_{s}}), \Theta_{s}\rangle_{\mathcal{V}}+\|B(\Theta_{s},\mathscr{L}_{\Theta_{s}})\|_{L_2(\mathcal{U},\mathcal{H})}^2-\lambda\|\Theta_{s}\|_{\mathcal{H}}^2\right)ds\right].
\end{eqnarray}
Inserting (\ref{16}) into (\ref{7}) and rearranging it implies
\begin{eqnarray*}
\!\!\!\!\!\!\!\!&&\mathbb{E}\Big\{\int_0^T\psi_t\Big[\int_0^te^{-\lambda s}\Big(2{}_{\mathcal{V^*}}\langle\bar{F}_s-{A}(\phi_s,\mathscr{L}_{\phi_s}),\Theta_{s}-\phi_s\rangle_{\mathcal{V}}+\|\bar{Z}_s-B(\phi_s,\mathscr{L}_{\phi_s})\|_{L_2(\mathcal{U},\mathcal{H})}^2
\nonumber\\
\!\!\!\!\!\!\!\!&&~~~~~~~~~~~~~~~~~~~~~~~~~~-\lambda\|\Theta_s-\phi_s\|_{\mathcal{H}}^2\Big)ds\Big]dt\Big\}\leq 0.
\end{eqnarray*}
First, taking $\phi=\Theta$ implies that $B(\Theta_{\cdot},\mathscr{L}_{\Theta_{\cdot}})=\bar{Z}$.
Next, letting $\phi=\Theta-\eta\tilde{\phi}v$ for any $\eta>0$, $v\in \mathcal{V}$ and $\tilde{\phi}\in L^\infty([0,T]\times\Omega,dt\times\mathbb{P};\mathbb{R})$. It follows that
$$\mathbb{W}_{2,\mathcal{H}}(\mathscr{L}_{\Theta_{s}},\mathscr{L}_{\phi_s})^2\leq\mathbb{E}\|\eta\tilde{\phi}_sv\|_{\mathcal{H}}^2\leq\eta\|\tilde{\phi}\|_{\infty}^2\|v\|_{\mathcal{H}}^2\downarrow0,~\text{as}~\eta\downarrow0.$$
Then taking $\eta\to 0$ by dominated convergence theorem we have
\begin{eqnarray*}
\!\!\!\!\!\!\!\!&&\mathbb{E}\Big\{\int_0^T\psi_t\Big[\int_0^te^{-\lambda s}{}_{\mathcal{V^*}}\langle\bar{F}_s-{A}(\Theta_s,\mathscr{L}_{\Theta_s}),\tilde{\phi}_sv\rangle_{\mathcal{V}}ds\Big]dt\Big\}\leq 0.
\end{eqnarray*}
The converse follows by taking $\tilde{\phi}=-\tilde{\phi}$, which concludes ${A}(\Theta_{\cdot},\mathscr{L}_{\Theta_{\cdot}})=\bar{F}$.

The uniqueness of solutions to systems \eref{e2} follows from the It\^{o}'s formula, \eref{mon} and \eref{Lip} directly. Hence we complete the proof of Theorem \ref{Th1}.
\hspace{\fill}$\Box$

\section{Proof of Averaging principle} \label{Sec Proof of Thm1}
\setcounter{equation}{0}
 \setcounter{definition}{0}
In this section, we aim to prove that the slow component of system \eref{e1} strongly converges to the solution of the corresponding averaged equation, which is mainly based on the technique of Khasminskii time discretization. In particular, the corresponding convergence rate is also derived.
\subsection{Some apriori estimates for system (\ref{e1})}
We first give some uniform bounds with respect to $\vare\in (0,1)$ for the solutions $(X_{t}^{\varepsilon}, Y_{t}^{\vare})$ of  system \eref{e1}.
\begin{lemma} \label{PMY}
For any $T>0$, there exists a constant $C_{T}>0$ such that,
\begin{align}
\sup_{\varepsilon\in(0,1)}\mathbb{E}\left(\sup_{t\in[0,T]}\|X_{t}^{\vare}\|_{H_1}^{4}\right)+\sup_{\varepsilon\in(0,1)}\EE\left(\int^T_0\|X_{t}^{\vare}\|_{V_1}^\alpha dt\right)\leq C_{T}\left(1+\|x\|_{H_1}^4+\|y\|_{H_2}^4\right)\label{F3.1}
\end{align}
and
\begin{align}\label{8}
\sup_{\varepsilon\in(0,1)}\sup_{t\in[0, T]}\mathbb{E}\|Y_{t}^{\varepsilon}\|_{H_2}^{4}\leq C_{T}\left(1+\|x\|_{H_1}^4+\|y\|_{H_2}^4\right).
\end{align}
\end{lemma}
\begin{proof}
Applying It\^{o}'s formula for $\|Y^{\varepsilon}_t\|_{H_2}^4$, we have
\begin{eqnarray*}
\|Y^{\varepsilon}_t\|_{H_2}^4
=\!\!\!\!\!\!\!\!&&\|y\|_{H_2}^4+\frac{4}{\varepsilon}\int_0^t\|Y^{\varepsilon}_s\|_{H_2}^2{}_{V_2^*}\langle A_2(X^{\varepsilon}_s,\mathscr{L}_{X^{\varepsilon}_s},Y^{\varepsilon}_s),Y^{\varepsilon}_s\rangle_{V_2}ds
\nonumber\\
\!\!\!\!\!\!\!\!&&+\frac{4}{\varepsilon}\int_0^t\|B_2(X^{\varepsilon}_s,\mathscr{L}_{X^{\varepsilon}_s},Y^{\varepsilon}_s)^*Y^{\varepsilon}_s\|_{U_2}^2ds+\frac{2}{\varepsilon}\int_0^t\|Y^{\varepsilon}_s\|_{H_2}^2\|B_2(X^{\varepsilon}_s,\mathscr{L}_{X^{\varepsilon}_s},Y^{\varepsilon}_s)\|_{L_2(U_2,H_2)}^2ds
\nonumber\\
\!\!\!\!\!\!\!\!&&
+\frac{4}{\varepsilon}\int_0^t\|Y^{\varepsilon}_s\|_{H_2}^2\langle B_2(X^{\varepsilon}_s,\mathscr{L}_{X^{\varepsilon}_s},Y^{\varepsilon}_s)dW^2_s,Y^{\varepsilon}_s\rangle_{H_2}.
\end{eqnarray*}
Following the same calculations as in the proof of \cite[Lemma 4.3.8]{LR1}, by Hypothesis \ref{h2}, there is a constant $\lambda\in(0,\kappa)$ such that for any $u\in H_1,\mu\in\mathscr{P}_2(H_1)$ and $v\in V_2$,
\begin{eqnarray}\label{a2}
2{}_{V_2^*}\langle A_2(u,\mu,v),v\rangle_{V_2}+\|B_2(u,\mu,v)\|_{L_2(U_2,H_2)}^2\leq-\lambda\|v\|_{H_2}^2+C\left(1+\|u\|_{H_1}^2+\mu(\|\cdot\|_{H_1}^2)\right).
\end{eqnarray}
Taking expectation and differentiating with respect to $t$, by \eref{a2} we deduce that
\begin{eqnarray*}
\frac{d}{dt}\mathbb{E}\|Y^{\varepsilon}_t\|_{H_2}^4=\!\!\!\!\!\!\!\!&&\frac{4}{\varepsilon}\mathbb{E}\left(\|Y^{\varepsilon}_t\|_{H_2}^2{}_{V_2^*}\langle A_2(X^{\varepsilon}_t,\mathscr{L}_{X^{\varepsilon}_t},Y^{\varepsilon}_t),Y^{\varepsilon}_t\rangle_{V_2}\right) \nonumber \\
 \!\!\!\!\!\!\!\!&& +\frac{4}{\varepsilon}\mathbb{E}\left(\|B_2(X^{\varepsilon}_t,\mathscr{L}_{X^{\varepsilon}_t},Y^{\varepsilon}_t)^*Y^{\varepsilon}_t\|_{U_2}^2\right)
+\frac{2}{\varepsilon}\mathbb{E}\left(\|Y^{\varepsilon}_t\|_{H_2}^2\|B_2(X^{\varepsilon}_t,\mathscr{L}_{X^{\varepsilon}_t},Y^{\varepsilon}_t)\|_{L_2(U_2,H_2)}^2\right) \\
 \leq\!\!\!\!\!\!\!\!&&\frac{2}{\varepsilon}\mathbb{E}\left[\|Y^{\varepsilon}_t\|_{H_2}^2\left({2}_{V_2^*}\langle A_2(X^{\varepsilon}_t,\mathscr{L}_{X^{\varepsilon}_t},Y^{\varepsilon}_t),Y^{\varepsilon}_t\rangle_{V_2}+3\|B_2(X^{\varepsilon}_t,\mathscr{L}_{X^{\varepsilon}_t},Y^{\varepsilon}_t)\|_{L_2(U_2,H_2)}^2\right)\right] \\
\leq\!\!\!\!\!\!\!\!&&\frac{2}{\varepsilon}\mathbb{E}\left[\|Y^{\varepsilon}_t\|_{H_2}^2\left(-\lambda\|Y^{\varepsilon}_t\|_{H_2}^2+C\|X^{\varepsilon}_t\|_{H_1}^2+C\mathscr{L}_{X^{\varepsilon}_t}(\|\cdot\|_{H_1}^2)+C\right)\right] \\
\leq\!\!\!\!\!\!\!\!&&-\frac{2\lambda_0}{\varepsilon}\mathbb{E}\|Y_{t}^{\varepsilon}
\|_{H_2}^{4}+\frac{C}{\varepsilon}\EE\|X_{t}^{\varepsilon}\|_{H_1}^{4}
+\frac{C}{\varepsilon},\label{4.4.2}
\end{eqnarray*}
where $\lambda_0\in(0,\lambda)$ and we used the fact that $\mathscr{L}_{X^{\varepsilon}_t}(\|\cdot\|_{H_1}^2)=\EE\|X_{t}^{\varepsilon}\|_{H_1}^{2}$. Hence, by the comparison theorem, it is easy to see that
\begin{eqnarray}
\mathbb{E}\|Y_{t}^{\varepsilon}
\|_{H_2}^{4}\leq\!\!\!\!\!\!\!\!&&\|y\|_{H_2}^{4}e^{-\frac{2\lambda_0}{\varepsilon}t}+\frac{C}{\varepsilon}\int^t_0
e^{-\frac{2\lambda_0}{\varepsilon}(t-s)}\left(1+\EE\|X_{s}^{\varepsilon}\|_{H_1}^{4}\right)ds.\label{F3}
\end{eqnarray}
On the other hand, using It\^{o}'s formula again, we also have
\begin{eqnarray*}
\|X^{\varepsilon}_t\|_{H_1}^4
=\!\!\!\!\!\!\!\!&&\|x\|_{H_1}^4+4\int_0^t\|X^{\varepsilon}_s\|_{H_1}^2{}_{V_1^*}\langle A_1(X^{\varepsilon}_s,\mathscr{L}_{X^{\varepsilon}_s})+f(X^{\varepsilon}_s,\mathscr{L}_{X^{\varepsilon}_s},Y^{\varepsilon}_s),X^{\varepsilon}_s\rangle_{V_1}ds
\nonumber\\
\!\!\!\!\!\!\!\!&&+4\int_0^t\|B_1(X^{\varepsilon}_s,\mathscr{L}_{X^{\varepsilon}_s})^*X^{\varepsilon}_s\|_{U_1}^2ds+2\int_0^t\|X^{\varepsilon}_s\|_{H_1}^2\|B_1(X^{\varepsilon}_s,\mathscr{L}_{X^{\varepsilon}_s})\|_{L_2(U_1,H_1)}^2ds
\nonumber\\
\!\!\!\!\!\!\!\!&&
+4\int_0^t\|X^{\varepsilon}_s\|_{H_1}^2\langle B_1(X^{\varepsilon}_s,\mathscr{L}_{X^{\varepsilon}_s})dW^1_s,X^{\varepsilon}_s\rangle_{H_1}ds.
\end{eqnarray*}
Then by Burkholder-Davis-Gundy's inequality, \eref{F3} and Hypothesis \ref{h1}, it holds that
\begin{eqnarray*}
&&\mathbb{E}\left(\sup_{t\in[0, T]}\|X_{t}^{\vare}\|_{H_1}^{4}\right)+4\EE\left(\int^T_0\|X_{t}^{\vare}\|_{H_1}^{2}\|X_{t}^{\vare}\|_{V_1}^\alpha dt\right)\\
\leq\!\!\!\!\!\!\!\!&&{\|x\|_{H_1}^{4}}+C_T+C\int^T_0\mathbb{E}\|X_{t}^{\varepsilon}\|_{H_1}^{4}dt+C\int^T_0\mathbb{E}\| Y_{t}^{\varepsilon}\|_{H_2}^{4}dt\\
\leq\!\!\!\!\!\!\!\!&&C_{T}\left(1+\|x\|_{H_1}^4+\|y\|_{H_2}^4\right)+C\int^T_0\mathbb{E}\|X_{t}^{\varepsilon}\|_{H_1}^{4}dt
\\\!\!\!\!\!\!\!\!&&+\frac{C}{\vare}\int^T_0\int^t_0
e^{-\frac{2\lambda_0}{\varepsilon}(t-s)}\left(1+\EE\|X_{s}^{\varepsilon}\|_{H_1}^{4}\right)dsdt\\
\leq\!\!\!\!\!\!\!\!&&C_{T}\left(1+\|x\|_{H_1}^4+\|y\|_{H_2}^4\right)+C\int^T_0\mathbb{E}\|X_{t}^{\varepsilon}\|_{H_1}^{4}dt.
\end{eqnarray*}
Hence, applying Gronwall's inequality, we get
\begin{eqnarray}
\mathbb{E}\left(\sup_{t\in[0, T]}\|X_{t}^{\vare}\|_{H_1}^{4}\right)+\EE\left(\int^T_0\|X_{t}^{\vare}\|_{H_1}^{2}\|X_{t}^{\vare}\|_{V_1}^\alpha dt\right)
\leq\!\!\!\!\!\!\!\!&&C_{T}\left(1+\|x\|_{H_1}^4+\|y\|_{H_2}^4\right),\label{4.4.4}
\end{eqnarray}
which also gives
\begin{eqnarray*}
\mathbb{E}\|Y_{t}^{\varepsilon}\|_{H_2}^{4}\leq
C_{T}\left(1+\|x\|_{H_1}^4+\|y\|_{H_2}^4\right).\label{4.4.5}
\end{eqnarray*}
Moreover, applying It\^{o}'s formula to $\|X^{\varepsilon}_t\|_{H_1}^2,\|Y^{\varepsilon}_t\|_{H_2}^2$ and following the same procedure as (\ref{4.4.4}), it is obvious that
\begin{eqnarray*}
\EE\left(\int^T_0\|X_{t}^{\vare}\|_{V_1}^\alpha dt\right)\leq C_{T}\left(1+\|x\|_{H_1}^4+\|y\|_{H_2}^4\right).
\end{eqnarray*}

The proof is complete.\hspace{\fill}$\Box$
\end{proof}

\vspace{0.1cm}
The following Lemma is an estimate of the integral of the time increment of $X_{t}^{\varepsilon}$, which is weaker than the H\"{o}lder continuity of time (see e.g. \cite{DSXZ,FWL,FWLL}) but strong enough for our purpose, and the advantage is it only needs initial value $x\in H_1,y\in H_2$.
\begin{lemma} \label{COX}
For any $T>0$, there exists a constant $C_{T}>0$ such that for any $\vare\in(0,1)$ and $\delta>0$ small enough,
\begin{align}
\mathbb{E}\left[\int^{T}_0\|X_{t}^{\varepsilon}-X_{t(\delta)}^{\varepsilon}\|_{H_1}^2 dt\right]\leq C_{T}\delta(1+\|x\|_{H_1}^2+\|y\|_{H_2}^2),\label{F3.7}
\end{align}
where $t(\delta):=[\frac{t}{\delta}]\delta$ and $[s]$ denotes the
integer part of $s$.
\end{lemma}

\begin{proof}
Using \eref{F3.1}, it is easy to get that
\begin{eqnarray}
&&\mathbb{E}\left[\int^{T}_0\|X_{t}^{\varepsilon}-X_{t(\delta)}^{\varepsilon}\|_{H_1}^2dt\right]\nonumber\\
=\!\!\!\!\!\!\!\!&& \mathbb{E}\left(\int^{\delta}_0\|X_{t}^{\varepsilon}-x\|_{H_1}^2dt\right)+\mathbb{E}\left[\int^{T}_{\delta}\|X_{t}^{\varepsilon}-X_{t(\delta)}^{\varepsilon}\|_{H_1}^2dt\right]\nonumber\\
\leq\!\!\!\!\!\!\!\!&& C\left(1+\|x\|_{H_1}^2+\|y\|_{H_2}^2\right)\delta \nonumber \\
 \!\!\!\!\!\!\!\!&& +2\mathbb{E}\left(\int^{T}_{\delta}\|X_{t}^{\varepsilon}-X_{t-\delta}^{\varepsilon}\|_{H_1}^2dt\right)+2\mathbb{E}\left(\int^{T}_{\delta}\|X_{t(\delta)}^{\varepsilon}-X_{t-\delta}^{\varepsilon}\|_{H_1}^2dt\right).\label{F3.8}
\end{eqnarray}
It follows from It\^{o}'s formula that
\begin{eqnarray}
\|X_{t}^{\varepsilon}-X_{t-\delta}^{\varepsilon}\|_{H_1}^{2}=\!\!\!\!\!\!\!\!&&2\int_{t-\delta} ^{t}{}_{V_1^*}\langle A_1(X^{\varepsilon}_s,\mathscr{L}_{X^{\varepsilon}_s}), X_{s}^{\varepsilon}-X_{t-\delta}^{\varepsilon}\rangle_{V_1} ds
\nonumber \\
 \!\!\!\!\!\!\!\!&&+ 2\int_{t-\delta} ^{t}{}_{V_1^*}\langle f(X^{\varepsilon}_s,\mathscr{L}_{X^{\varepsilon}_s},Y^{\varepsilon}_s), X_{s}^{\varepsilon}-X_{t-\delta}^{\varepsilon}\rangle_{V_1} ds\nonumber \\
 \!\!\!\!\!\!\!\!&& +\int_{t-\delta} ^{t}\|B_1(X^{\varepsilon}_s,\mathscr{L}_{X^{\varepsilon}_s})\|_{L_2(U_1,H_1)}^2ds
 +2\int_{t-\delta} ^{t}\langle B_1(X^{\varepsilon}_s,\mathscr{L}_{X^{\varepsilon}_s})dW^{1}_s, X_{s}^{\varepsilon}-X_{t-\delta}\rangle_{H_1} \nonumber\\
:=\!\!\!\!\!\!\!\!&&I_{1}(t)+I_{2}(t)+I_{3}(t)+I_{4}(t).  \label{F3.9}
\end{eqnarray}
For the first term $I_1(t)$, by condition $({\mathbf{A}}{\mathbf{4}})$, there exists a constant $C>0$ such that
\begin{eqnarray}  \label{REGX1}
&&\mathbb{E}\left(\int^{T}_{\delta}I_{1}(t)dt\right)\nonumber\\
\leq\!\!\!\!\!\!\!\!&& C\mathbb{E}\left(\int^{T}_{\delta}\int_{t-\delta} ^{t}\|A_1(X^{\varepsilon}_s,\mathscr{L}_{X^{\varepsilon}_s})\|_{V_1^*}
\|X_{s}^{\varepsilon}-X_{t-\delta}^{\varepsilon}\|_{V_1} ds dt\right)\nonumber\\
\leq\!\!\!\!\!\!\!\!&&C\left[\mathbb{E}\int^{T}_{\delta}\int_{t-\delta} ^{t}\|A_1(X^{\varepsilon}_s,\mathscr{L}_{X^{\varepsilon}_s})\|_{V_1^*}^{\alpha/(\alpha-1)}dsdt\right]^{(\alpha-1)/\alpha}
\left[\mathbb{E}\int^{T}_{\delta}\int_{t-\delta} ^{t}\|X_{s}^{\varepsilon}-X_{t-\delta}^{\varepsilon}\|^\alpha_{V_1} dsdt\right]^{1/\alpha}\nonumber\\
\leq\!\!\!\!\!\!\!\!&&C\left[\delta\mathbb{E}\int^{T}_0(1+\|X_{s}^{\varepsilon}\|_{V_1}^\alpha+\mathscr{L}_{X^{\varepsilon}_s}(\|\cdot\|_{H_1}^2))ds\right]^{(\alpha-1)/\alpha}
\cdot\left[\delta\mathbb{E}\int^{T}_0\|X_{s}^{\varepsilon}\|^\alpha_{V}ds\right]^{1/\alpha}\nonumber\\
\leq\!\!\!\!\!\!\!\!&&C_{T}\delta\left(1+\|x\|_{H_1}^2+\|y\|_{H_2}^2\right),
\end{eqnarray}
where we use Fubini's theorem and \eref{F3.1} in the third and fourth inequalities respectively.

For $I_{2}(t)$ and $I_3(t)$, by condition $({\mathbf{A}}{\mathbf{2}})$, \eref{F3.1} and (\ref{8}), we get
\begin{eqnarray}\label{REGX2}
&&\mathbb{E}\left(\int^{T}_{\delta}I_{2}(t)dt\right)\nonumber\\
\leq\!\!\!\!\!\!\!\!&&C\mathbb{E}\left[\int^{T}_{\delta}\int_{t-\delta} ^{t}\left(1+\|X_{s}^{\varepsilon}\|_{H_1}+\|Y_{s}^{\varepsilon}\|_{H_2}+\big(\mathscr{L}_{X^{\varepsilon}_s}(\|\cdot\|_{H_1}^2)\big)^{1/2}\right)\left(\|X_{s}^{\varepsilon}\|_{H_1}+\|X_{t-\delta}^{\varepsilon}\|_{H_1}\right)ds dt\right]\nonumber\\
\leq\!\!\!\!\!\!\!\!&&C_T\delta\mathbb{E}\left[\sup_{s\in[0,T]}(1+\|X_{s}^{\varepsilon}\|_{H_1}^2)\right]+C_T\delta\left(\int_0^T\mathbb{E}\|X_{s}^{\varepsilon}\|_{H_1}^2ds\right)^{1/2}\left[\mathbb{E}\left(\sup_{s\in[0,T]}\|X_{s}^{\varepsilon}\|_{H_1}^2\right)\right]^{1/2}
\nonumber\\\!\!\!\!\!\!\!\!&&+C\mathbb{E}\left[\sup_{s\in[0,T]}\|X_{s}^{\varepsilon}\|_{H_1}\int^T_{\delta}\int^t_{t-\delta}\|Y^{\vare}_s\|_{H_2}dsdt\right]\nonumber\\
\leq\!\!\!\!\!\!\!\!&&C_T\delta\mathbb{E}\left[\sup_{s\in[0,T]}(1+\|X_{s}^{\varepsilon}\|_{H_1}^2)\right]+C_T\delta^{1/2}\left[\mathbb{E}\left(\sup_{s\in[0,T]}\|X_{s}^{\varepsilon}\|_{H_1}^2\right)\right]^{1/2}\!\!\!\!\!\left[\mathbb{E}\left(\int^T_{\delta}\int^t_{t-\delta}\|Y_{s}^{\varepsilon}\|_{H_2}^2dsdt\right)\right]^{1/2}\nonumber\\
\leq\!\!\!\!\!\!\!\!&&C_{T}\delta\mathbb{E}\left[\sup_{s\in[0,T]}(1+\|X_{s}^{\varepsilon}\|_{H_1}^2)\right]+C_T\delta\int_0^T\mathbb{E}\|Y_{s}^{\varepsilon}\|_{H_2}^2ds
\nonumber\\\leq\!\!\!\!\!\!\!\!&&C_{T}\delta\left(1+\|x\|_{H_1}^2+\|y\|_{H_2}^2\right)
\end{eqnarray}
and
\begin{eqnarray}\label{REGX2a}
\mathbb{E}\left(\int^{T}_{\delta}I_{3}(t)dt\right)\leq\!\!\!\!\!\!\!\!&&C\mathbb{E}\left[\int^{T}_{\delta}\int_{t-\delta} ^{t}\left(1+\|X_{s}^{\varepsilon}\|_{H_1}^2+\mathscr{L}_{X^{\varepsilon}_s}(\|\cdot\|_{H_1}^2)\right)ds dt\right]\nonumber\\
\leq\!\!\!\!\!\!\!\!&&C_T\delta\mathbb{E}\left[\sup_{s\in[0,T]}\left(1+\|X_{s}^{\varepsilon}\|_{H_1}^2+\EE\|X_{s}^{\varepsilon}\|_{H_1}^2\right)\right]\nonumber\\
\leq\!\!\!\!\!\!\!\!&&C_{T}\delta\left(1+\|x\|_{H_1}^2+\|y\|_{H_2}^2\right).
\end{eqnarray}
For $I_{4}(t)$, due to Lemma \ref{PMY}, it is easy to see that
\begin{eqnarray}  \label{REGX3}
\mathbb{E}\left(\int^{T}_{\delta}I_{4}(t)dt\right)=\!\!\!\!\!\!\!\!&&\int^{T}_{\delta}\mathbb{E}\left[\int_{t-\delta} ^{t}\langle B_1(X^{\varepsilon}_s,\mathscr{L}_{X^{\varepsilon}_s})dW^1_s,X_{s}^{\varepsilon}-X_{t-\delta}^{\varepsilon}\rangle_{H_1}\right]dt \nonumber\\
=\!\!\!\!\!\!\!\!&&0.
\end{eqnarray}
Combining estimates \eref{F3.9}-\eref{REGX3}, we get that
\begin{eqnarray}
\mathbb{E}\left(\int^{T}_{\delta}\|X_{t}^{\varepsilon}-X_{t-\delta}^{\varepsilon}\|_{H_1}^2dt\right)\leq\!\!\!\!\!\!\!\!&&C_{T}\delta(1+\|x\|_{H_1}^2+\|y\|_{H_2}^2). \label{F3.13}
\end{eqnarray}
By a similar argument as above, we can also get
\begin{eqnarray}
\mathbb{E}\left(\int^{T}_{\delta}\|X_{t(\delta)}^{\varepsilon}-X_{t-\delta}^{\varepsilon}\|_{H_1}^2dt\right)\leq\!\!\!\!\!\!\!\!&&C_{T}\delta(1+\|x\|_{H_1}^2+\|y\|_{H_2}^2). \label{F3.14}
\end{eqnarray}
Hence, \eref{F3.8}, \eref{F3.13} and \eref{F3.14} implies \eref{F3.7} holds. The proof is complete.\hspace{\fill}$\Box$
\end{proof}

\subsection{Estimates of auxiliary process}
Inspired by the time discretization method developed in \cite{K1}, we
divide $[0,T]$ into intervals of size $\delta$, where $\delta$ is a fixed positive number depending on $\vare$ and will be chosen later.
Then, we
construct an auxiliary process $\hat{Y}_{t}^{\varepsilon}\in{H_2}$, with $\hat{Y}_{0}^{\varepsilon}=Y^{\varepsilon}_{0}=y$, and for any $k\in \mathbb{N}$ and $t\in[k\delta,\min((k+1)\delta,T)]$,
\begin{eqnarray}
\hat{Y}_{t}^{\varepsilon}=\hat{Y}_{k\delta}^{\varepsilon}+\frac{1}{\varepsilon}\int_{k\delta}^{t}
A_2(X_{k\delta}^{\varepsilon},\mathscr{L}_{X^{\vare}_{k\delta}},\hat{Y}_{s}^{\varepsilon})ds+\frac{1}{\sqrt{\varepsilon}}\int_{k\delta}^{t}B_2(X_{k\delta}^{\varepsilon},\mathscr{L}_{X^{\vare}_{k\delta}},\hat{Y}_{s}^{\varepsilon})dW^{{2}}_s,\label{4.6a}
\end{eqnarray}
which is equivalent to
$$
d\hat{Y}_{t}^{\vare}=\frac{1}{\vare}\left[A_2\left(X^{\vare}_{t(\delta)},\mathscr{L}_{X^{\vare}_{t(\delta)}},\hat{Y}_{t}^{\vare}\right)\right]dt+\frac{1}{\sqrt{\vare}}B_2\left(X^{\vare}_{t(\delta)},\mathscr{L}_{X^{\vare}_{t(\delta)}},\hat{Y}_{t}^{\vare}\right)dW^2_t,\quad \hat{Y}_{0}^{\vare}=y.
$$

\vspace{0.2cm}
By the construction of $
\hat{Y}_{t}^{\varepsilon}$, we can obtain the following
estimates which will be used below.

\begin{lemma} \label{MDYa}
For any $T>0$, there exists a constant
$C_{T}>0$ such that,
\begin{eqnarray}
\sup_{\varepsilon\in(0,1)}\sup_{t\in[0,T]}\mathbb{E}\|\hat{Y}_{t}^{\vare}\|_{H_2}^2\leq
C_{T}\left(1+\|x\|_{H_1}^2+\|y\|_{H_2}^2\right) \label{3.13a}
\end{eqnarray}
and
\begin{eqnarray}
\sup_{\varepsilon\in(0,1)}\mathbb{E}\left(\int_0^{T}\|Y_{t}^{\varepsilon}-\hat{Y}_{t}^{\varepsilon}\|_{H_2}^2dt\right)\leq C_{T}\delta(1+\|x\|_{H_1}^2+\|y\|_{H_2}^2). \label{3.14}
\end{eqnarray}
\end{lemma}

\begin{proof}
Since the proof of \eref{3.13a} is similar to
 Lemma \ref{PMY}, we omit it here. Next, we will prove \eref{3.14}.
It is easy to see that $Y_{t}^{\varepsilon}-\hat{Y}_{t}^{\varepsilon}$ satisfies the following equation
\begin{eqnarray}\label{e7}
\left\{ \begin{aligned}
&d(Y_{t}^{\varepsilon}-\hat{Y}_{t}^{\varepsilon})=\frac{1}{\varepsilon}\left[A_2\left(X^{\vare}_{t},\mathscr{L}_{X^{\vare}_{t}},{Y}_{t}^{\vare}\right)-A_2\left(X^{\vare}_{t(\delta)},\mathscr{L}_{X^{\vare}_{t(\delta)}},\hat{Y}_{t}^{\vare}\right)\right]dt\\ &~~~~~~~~~~~~~~~~~+\frac{1}{\sqrt{\vare}}\left[B_2\left(X^{\vare}_{t},\mathscr{L}_{X^{\vare}_{t}},{Y}_{t}^{\vare}\right)-B_2\left(X^{\vare}_{t(\delta)},\mathscr{L}_{X^{\vare}_{t(\delta)}},\hat{Y}_{t}^{\vare}\right)\right]dW_t^{2},\\
&Y_{0}^{\varepsilon}-\hat{Y}_{0}^{\varepsilon}=0,
\end{aligned}\right.
\end{eqnarray}

Thus, applying It\^{o}'s formula and taking expectation, we get
\begin{eqnarray*}
\EE\|Y_{t}^{\varepsilon}-\hat{Y}_{t}^{\varepsilon}\|_{H_2}^2
=\!\!\!\!\!\!\!\!&&\frac{2}{\varepsilon}\EE\int^t_0{}_{V_2^*}\langle A_2\left(X^{\vare}_{s},\mathscr{L}_{X^{\vare}_{s}},{Y}_{s}^{\vare}\right)-A_2\left(X^{\vare}_{s(\delta)},
\mathscr{L}_{X^{\vare}_{s(\delta)}},\hat{Y}_{s}^{\vare}\right),Y_{s}^{\varepsilon}-\hat{Y}_{s}^{\varepsilon}\rangle_{V_2} ds \nonumber\\
\!\!\!\!\!\!\!\!&&+\frac{1}{\varepsilon}\EE\int^t_0\|B_2\left(X^{\vare}_{s},\mathscr{L}_{X^{\vare}_{s}},{Y}_{s}^{\vare}\right)-
B_2\left(X^{\vare}_{s(\delta)},\mathscr{L}_{X^{\vare}_{s(\delta)}},\hat{Y}_{s}^{\vare}\right)\|_{L_2(U_2,H_2)}^2ds.
\end{eqnarray*}
Then by condition $({\mathbf{H}}{\mathbf{2}})$, we have
\begin{eqnarray}\label{f1}
\!\!\!\!\!\!\!\!&&\frac{d}{dt}\EE\|Y_{t}^{\varepsilon}-\hat{Y}_{t}^{\varepsilon}\|_{H_2}^2
\nonumber\\=\!\!\!\!\!\!\!\!&&\frac{2}{\varepsilon}\EE{}_{V_2^*}\langle A_2\left(X^{\vare}_{t},\mathscr{L}_{X^{\vare}_{t}},{Y}_{t}^{\vare}\right)-A_2\left(X^{\vare}_{t(\delta)},\mathscr{L}_{X^{\vare}_{t(\delta)}},\hat{Y}_{t}^{\vare}\right),Y_{t}^{\varepsilon}-\hat{Y}_{t}^{\varepsilon}\rangle_{V_2} \nonumber\\
\!\!\!\!\!\!\!\!&&+\frac{1}{\varepsilon}\EE\left\|B_2\left(X^{\vare}_{t},\mathscr{L}_{X^{\vare}_{t}},{Y}_{t}^{\vare}\right)-B_2\left(X^{\vare}_{t(\delta)},\mathscr{L}_{X^{\vare}_{t(\delta)}},\hat{Y}_{t}^{\vare}\right)\right\|_{L_2(U_2,H_2)}^2\nonumber\\
\leq\!\!\!\!\!\!\!\!&&-\frac{2\kappa}{\varepsilon}\EE\|Y_{t}^{\varepsilon}-\hat{Y}_{t}^{\varepsilon}\|_{H_2}^2
+\frac{c_2}{\varepsilon}\EE\Big[\|X_t^\varepsilon-X_{t(\delta)}^\varepsilon\|_{H_1}^2+\mathbb{W}_{2,H_1}(\mathscr{L}_{X^{\vare}_{t}},\mathscr{L}_{X^{\vare}_{t(\delta)}})^2\Big]
  \nonumber\\
\!\!\!\!\!\!\!\!&&+\frac{1}{\varepsilon}\EE\Big[L_{B_2}\|Y_{t}^{\varepsilon}-\hat{Y}_{t}^{\varepsilon}\|_{H_2}
+c_2\|X_t^\varepsilon-X_{t(\delta)}^\varepsilon\|_{H_1}+c_2\mathbb{W}_{2,H_1}(\mathscr{L}_{X^{\vare}_{t}},\mathscr{L}_{X^{\vare}_{t(\delta)}})\Big]^2.
\end{eqnarray}
Note that
\begin{eqnarray}\label{f2}
\mathbb{W}_{2,H_1}(\mathscr{L}_{X^{\vare}_{t}},\mathscr{L}_{X^{\vare}_{t(\delta)}})^2\leq\EE\|X_t^\varepsilon-X_{t(\delta)}^\varepsilon\|_{H_1}^2.
\end{eqnarray}
Due to $2\kappa>L_{B_2}^2$, then according to  \eref{f1} and \eref{f2} there exists $\theta>0$ such that
\begin{eqnarray*}
\frac{d}{dt}\EE\|Y_{t}^{\varepsilon}-\hat{Y}_{t}^{\varepsilon}\|_{H_2}^2
\leq-\frac{\theta}{\varepsilon}\EE\|Y_{t}^{\varepsilon}-\hat{Y}_{t}^{\varepsilon}\|_{H_2}^2+\frac{C}{\varepsilon}\EE\|X_t^\varepsilon-X_{t(\delta)}^\varepsilon\|_{H_1}^2.\nonumber
\end{eqnarray*}
Therefore, by the comparison theorem we have
\begin{eqnarray*}
\EE\|Y_{t}^{\varepsilon}-\hat{Y}_{t}^{\varepsilon}\|_{H_2}^2\leq\frac{C}{\varepsilon}\int_0^te^{-\frac{\theta(t-s)}{\vare}}\EE\|X_s^\varepsilon-X_{s(\delta)}^\varepsilon\|_{H_1}^2ds.
\end{eqnarray*}
Using Fubini's theorem, we can get that for any $T>0$,
\begin{eqnarray*}
\EE\left(\int_0^T\|Y_{t}^{\varepsilon}-\hat{Y}_{t}^{\varepsilon}\|_{H_2}^2dt\right)\leq\!\!\!\!\!\!\!\!&& \frac{C}{\varepsilon}\int_0^T\int^t_0e^{-\frac{\beta(t-s)}{\vare}}\EE\|X_s^\varepsilon-X_{s(\delta)}^\varepsilon\|_{H_1}^2dsdt\nonumber\\
=\!\!\!\!\!\!\!\!&&  \frac{C}{\varepsilon}\EE\left[\int_0^T\|X_s^\varepsilon-X_{s(\delta)}^\varepsilon\|_{H_1}^2\left(\int^T_s e^{-\frac{\theta(t-s)}{\vare}}dt\right)ds\right]\nonumber\\
\leq\!\!\!\!\!\!\!\!&& C\EE\left(\int_0^T\|X_s^\varepsilon-X_{s(\delta)}^\varepsilon\|_{H_1}^2 ds\right).
\end{eqnarray*}
It follows from  Lemma \ref{COX} that
\begin{eqnarray*}
\mathbb{E}\left(\int_0^{T}\|Y_{t}^{\varepsilon}-\hat{Y}_{t}^{\varepsilon}\|_{H_2}^2dt\right)\leq C_{T}\delta(1+\|x\|_{H_1}^2+\|y\|_{H_2}^2).
\end{eqnarray*}
The proof is complete.\hspace{\fill}$\Box$
\end{proof}

\subsection{The frozen and averaged equations}
In this subsection, we first introduce the frozen equation associated with the fast equation for a fixed slow component $x\in {H_1}$ and $\mu\in\mathscr{P}_2(H_1)$, i.e.,
\begin{eqnarray}
\left\{ \begin{aligned}
&dY_{t}=[A_2(x,\mu,Y_{t})]dt+B_2(x,\mu,Y_t)d\tilde{{W}}_{t}^{2},\\
&Y_{0}=y\in H_2,
\end{aligned} \right.\label{FEQ1}
\end{eqnarray}
where $\tilde{W}_{t}^{2}$ is a cylindrical Wiener process in a separable Hilbert space $U_2$ on another
probability space $(\tilde{\Omega},\tilde{\mathscr{F}},\tilde{\mathbb{P}})$ with natural filtration $(\tilde{\mathscr{F}}_t)_{t\geq 0}$.

 Since $x$ and $\mu$ are fixed in equation \eref{FEQ1}, following from \cite[Theorem 4.2.4]{LR1} under Hypothesis \ref{h2}, there is a unique solution denoted by $Y_t^{x,\mu,y}$ to equation (\ref{FEQ1}), which is a homogeneous Markov process. Let $P^{x,\mu}_t$ be the transition semigroup of $Y_{t}^{x,\mu,y}$,
that is, for any bounded measurable function $\varphi$ on $H_2$,
\begin{eqnarray*}
P^{x,\mu}_t \varphi(y)= \tilde{\mathbb{E}} \left[\varphi\left(Y_{t}^{x,\mu,y}\right)\right], \quad y \in H_2,\ \ t>0,
\end{eqnarray*}
where $\tilde{\mathbb{E}}$ is the expectation on $(\tilde{\Omega},\tilde{\mathscr{F}},\tilde{\mathbb{P}})$.
Then by \cite[Theorem 4.3.9]{LR1}, $P^{x,\mu}_t$ has a unique invariant measure $\nu^{x,\mu}$. Moreover, we have the following two propositions.
\begin{proposition}\label{Rem 4.0}
There exists a constant $C>0$ such that for any $x,x_1,x_2\in {H_1}, y\in {H_2}$ and $\mu,\nu\in\mathscr{P}_2(H_1)$,
\begin{equation}\label{9}
\sup_{t\in[0,\infty)}\tilde{\mathbb{E}}\|Y_t^{x,\mu,y}\|_{H_2}^2\leq C\Big(1+\|x\|_{H_1}+\|y\|_{H_2}+\mu(\|\cdot\|_{H_1}^2)\Big),
\end{equation}
\begin{equation}\label{10}
\sup_{t\in[0,\infty)}\tilde{\mathbb{E}}\|Y^{x_1,\mu,y}_t-Y^{x_2,\nu,y}_t\|_{H_2}^2\leq C\left(\|x_1-x_2\|_{H_1}^2+\mathbb{W}_{2,H_1}(\mu,\nu)^2\right).
\end{equation}
\end{proposition}
\begin{proof}
By It\^{o}'s formula, we have
\begin{eqnarray}\label{FF3.5}
\|Y^{x,\mu,y}_t\|_{H_2}^2
=\!\!\!\!\!\!\!\!&&\|y\|_{H_2}^2+2\int_0^t{}_{V_2^*}\langle A_2(x,\mu,Y^{x,\mu,y}_s),Y^{x,\mu,y}_s\rangle_{V_2}ds+\int_0^t\|B_2(x,\mu,Y^{x,\mu,y}_s)\|_{L_2(U_2,H_2)}^2ds
\nonumber\\
\!\!\!\!\!\!\!\!&&
+2\int_0^t\langle B_2(x,\mu,Y^{x,\mu,y}_s)d\tilde{W}^2_s,Y^{x,\mu,y}_s\rangle_{H_2}.
\end{eqnarray}
Taking
expectation on both sides of \eref{FF3.5}, by \eref{a2} we obtain
\begin{eqnarray*}
\frac{d}{dt}\tilde{\mathbb{E}}\|Y^{x,\mu,y}_t\|_{H_2}^2
=\!\!\!\!\!\!\!\!&&\tilde{\mathbb{E}}\left({2}_{V_2^*}\langle A_2(x,\mu,Y^{x,\mu,y}_t),Y^{x,\mu,y}_t\rangle_{V_2}+\|B_2(x,\mu,Y^{x,\mu,y}_t)\|_{L_2(U_2,H_2)}^2\right)
\nonumber\\
\leq\!\!\!\!\!\!\!\!&&-\lambda\|Y^{x,\mu,y}_t\|_{H_2}^2+C\left(1+\|x\|_{H_1}^2+\mu(\|\cdot\|_{H_1}^2)\right).
\end{eqnarray*}
Hence, applying the comparison theorem yields
\begin{eqnarray}
\tilde{\mathbb{E}}\|Y_{t}^{x,\mu,y}
\|_{H_2}^{2}\leq\!\!\!\!\!\!\!\!&&\|y\|_{H_2}^2e^{-\lambda t}+C\int^t_0
e^{-\lambda(t-s)}\left(1+\|x\|_{H_1}^2+\mu(\|\cdot\|_{H_1}^2)\right)ds
\nonumber\\
\leq\!\!\!\!\!\!\!\!&&\|y\|_{H_2}^2e^{-\lambda t}+C\left(1+\|x\|_{H_1}^2+\mu(\|\cdot\|_{H_1}^2)\right),\label{F3.4}
\end{eqnarray}
which gives (\ref{9}).

Following the similar calculations as above, by $({\mathbf{H}}{\mathbf{2}})$ it is obvious that (\ref{10}) holds. \hspace{\fill}$\Box$

\end{proof}

\begin{proposition}\label{Rem 4.1}
There exist $C>0$ and $\rho>0$ such that for any $x\in {H_1}, y\in {H_2}$ and $\mu\in\mathscr{P}_2(H_1)$,
\begin{equation}\label{11}
\Big\|\tilde{\mathbb{E}}f(x,\mu,Y_t^{x,\mu,y})-\bar{f}(x,\mu)\Big\|_{H_1}\leq C\Big(1+\|x\|_{H_1}+\|y\|_{H_2}+\mu(\|\cdot\|_{H_1}^2)^{1/2}\Big)e^{-\frac{\rho t}{2}},
\end{equation}
where $\bar{f}(x,\mu)=\int_{H_2}{f}(x,\mu,z)\nu^{x,\mu}(dz)$.
\end{proposition}

\begin{proof}
We denote by $Y_{t}^{x,\mu,y'}$ the solution of Eq.~$\eref{FEQ1}$ with initial value $Y_0=y'$.
Using It\^{o}'s formula and $({\mathbf{H}}{\mathbf{2}})$, similar to \eref{F3.4}, there exists a constant $\rho>0$ such that
\begin{eqnarray}
\tilde{\mathbb{E}}\|Y_{t}^{x,\mu,y}-Y_{t}^{x,\mu,y'}\|_{H_2}^{2}\leq\|y-y'\|_{H_2}^{2}e^{-\rho t},\label{FF3.4a}
\end{eqnarray}
for any $y,y'\in H_2$.

Then by the invariance of $\nu^{x,\mu}$ and
\eref{F3.4}, we have
\begin{eqnarray*}
\int_{H_2}\|y'\|_{H_2}^{2}\nu^{x,\mu}(dy')
=\!\!\!\!\!\!\!\!&&\int_{H_2}\tilde{\mathbb{E}}\|Y_{t}^{x,\mu,y'}\|_{H_2}^{2}\nu^{x,\mu}(dy')
\nonumber\\
\leq\!\!\!\!\!\!\!\!&& e^{-\lambda t}\int_{H_2}\|y'\|_{H_2}^{2}\nu^{x,\mu}(dy')+C\left(1+\|x\|_{H_1}^2+\mu(\|\cdot\|_{H_1}^2)\right).
\end{eqnarray*}
Take $t=t_0$ such that $e^{-\lambda t_0} < 1$, we have
\begin{eqnarray}
\int_{H_2}\|y'\|_{H_2}^{2}\nu^{x,\mu}(dy')
\leq C\left(1+\|x\|_{H_1}^2+\mu(\|\cdot\|_{H_1}^2)\right).\label{FF1}
\end{eqnarray}
Then, using the invariance of $\nu^{x,\mu}$, \eref{FF3.4a} and
\eref{FF1}, we have
\begin{eqnarray*}
\left\|\tilde{\mathbb{E}}f(x,\mu,Y_t^{x,\mu,y})-\bar{f}(x,\mu)\right\|_{H_1}
=\!\!\!\!\!\!\!\!&&\left\|\tilde{\mathbb{E}}f(x,\mu,Y_t^{x,\mu,y})-\int_{H_2}{f}(x,\mu,{y'})\nu^{x,\mu}(d{y'})\right\|_{H_1}
\nonumber\\=\!\!\!\!\!\!\!\!&&\left\|\int_{H_2} \big[\tilde{\mathbb{E}}f(x,\mu,Y_t^{x,\mu,y})-\tilde{\mathbb{E}}f(x,\mu,Y_t^{x,\mu,y'})\big]\nu^{x,\mu}(d{y'})\right\|_{H_1}
\nonumber\\\leq\!\!\!\!\!\!\!\!&&C\int_{H_2} \tilde{\mathbb{E}}\|Y_t^{x,\mu,y}-Y_t^{x,\mu,y'}\|_{H_2}\nu^{x,\mu}(d{y'})
\nonumber\\\leq\!\!\!\!\!\!\!\!&&Ce^{-\frac{\rho t}{2}}\int_{H_2} \|y-y'\|_{H_2}\nu^{x,\mu}(d{y'})
\nonumber\\\leq\!\!\!\!\!\!\!\!&&Ce^{-\frac{\rho t}{2}}\left(1+\|x\|_{H_1}+\|y\|_{H_2}+(\mu(\|\cdot\|_{H_1}^2))^{1/2}\right),
\end{eqnarray*}
which concludes the proof of Proposition \ref{Rem 4.1}.\hspace{\fill}$\Box$
\end{proof}

\vskip 0.3cm

Next, we consider the corresponding averaged equation, i.e.,
\begin{equation}\left\{\begin{array}{l}
\displaystyle d\bar{X}_{t}=A_1(\bar{X}_{t},\mathscr{L}_{\bar{X}_{t}})dt+\bar{f}(\bar{X}_{t},\mathscr{L}_{\bar{X}_{t}})dt+B_1(\bar{X}_{t},\mathscr{L}_{\bar{X}_{t}})dW^{{1}}_t,\\
\bar{X}_{0}=x,\end{array}\right. \label{3.1a}
\end{equation}
with
\begin{align*}
\bar{f}(x,\mu)=\int_{H_2}{f}(x,\mu,z)\nu^{x,\mu}(dz),
\end{align*}
with $\nu^{x,\mu}$ being the unique invariant measure of equation $\eref{FEQ1}$.
\vskip 0.3cm

\begin{Rem}
In terms of exponential ergodicity (\ref{11}) and Lipschitz of $f$, one can easily check that
$\bar{f}$ is also Lipschitz continuous, i.e.
\begin{eqnarray}\label{p0}
\|\bar{f}(x_1,\mu)-\bar{f}(x_2,\nu)\|_{H_1}
\leq C\left(\|x_1-x_2\|_{H_1}+\mathbb{W}_{2,H_1}(\mu,\nu)\right),x_1,x_2\in H_1,\mu,\nu\in\mathscr{P}_2(H_1).~
\end{eqnarray}

\begin{proof}
For any $x_1,x_2\in H_1$ and $\mu,\nu\in\mathscr{P}_2(H_1)$, by (\ref{10}) and (\ref{FF3.5}),
\begin{eqnarray}\label{12}
\!\!\!\!\!\!\!\!&&\|\bar{f}(x_1,\mu)-\bar{f}(x_2,\nu)\|_{H_1}^2
\nonumber\\
\leq\!\!\!\!\!\!\!\!&&C\left\|\bar{f}(x_1,\mu)-\tilde{\mathbb{E}}f(x_1,\mu,Y_t^{x_1,\mu,y})\right\|_{H_1}^2+C\left\|\bar{f}(x_2,\nu)-\tilde{\mathbb{E}}f(x_2,\nu,Y_t^{x_2,\nu,y})\right\|_{H_1}^2
\nonumber\\
\!\!\!\!\!\!\!\!&&
+C\left\|\tilde{\mathbb{E}}f(x_1,\mu,Y_t^{x_1,\mu,y})-\tilde{\mathbb{E}}f(x_2,\nu,Y_t^{x_2,\nu,y})\right\|_{H_1}^2
\nonumber\\
\leq\!\!\!\!\!\!\!\!&&C\left(1+\|x_1\|_{H_1}^2+\|x_2\|_{H_1}^2+\|y\|_{H_2}^2+\mu(\|\cdot\|_{H_1}^2)+\nu(\|\cdot\|_{H_1}^2)\right)e^{-\rho t}
\nonumber\\
\!\!\!\!\!\!\!\!&&
+C\left(\|x_1-x_2\|_{H_1}^2+\mathbb{W}_{2,H_1}(\mu,\nu)^2\right).
\end{eqnarray}
Taking $t\to\infty$ for both sides of (\ref{12}) leads to the desired estimate (\ref{p0}).\hspace{\fill}$\Box$
\end{proof}

\end{Rem}

Thus, similar to Theorem \ref{Th1}, for any $x\in H_1$, Eq.~\eref{3.1a} has a unique solution $\bar{X}_{t}$. Moreover, using an argument
similar to that in Lemma \ref{PMY} and Lemma \ref{COX}, we also have the following estimates.

\begin{lemma}\label{L3.8}
For any
$T>0$, there exists a constant $C_{T}>0$ such that
\begin{align*}
\mathbb{E}\left(\sup_{t\in[0,T]}\|\bar{X}_{t}\|_{H_1}^{2}\right)\leq C_{T}(1+\|x\|_{H_1}^2)
\end{align*}
and
\begin{align}
\mathbb{E}\left[\int^{T}_0\|\bar{X}_{t}-\bar{X}_{t(\delta)}\|_{H_1}^2 dt\right]\leq C_{T}\delta\left(1+\|x\|_{H_1}^2\right).\label{za1}
\end{align}
\end{lemma}

\medskip

Now we are in the position to finish the proof of the second main result.

\subsection{Proof of Theorem \ref{main result 1}}
The proof of Theorem \ref{main result 1} will be divided into the following three steps.

\textbf{Step 1}.
It is easy to see that $X_{t}^{\varepsilon}-\bar{X}_{t}$ satisfies the following equation
\begin{eqnarray*}\label{e5}
\left\{ \begin{aligned}
&d(X_{t}^{\varepsilon}-\bar{X}_{t})=\left[A_1\left(X^{\vare}_{t},\mathscr{L}_{X^{\vare}_{t}}\right)-A_1\left(\bar{X}_{t},\mathscr{L}_{\bar{X}_{t}}\right)+f\left(X^{\varepsilon}_t,\mathscr{L}_{X^{\varepsilon}_t},Y^{\varepsilon}_t\right)-\bar{f}\left(\bar{X}_{t},\mathscr{L}_{\bar{X}_{t}})\right)\right]dt\\ &~~~~~~~~~~~~~~~~~~+\left[B_1\left(X^{\vare}_{t},\mathscr{L}_{X^{\vare}_{t}}\right)-B_1\left(\bar{X}_{t},\mathscr{L}_{\bar{X}_{t}}\right)\right]dW_t^{1},\\
&X_{0}^{\varepsilon}-\bar{X}_{0}=0.
\end{aligned}\right.
\end{eqnarray*}
Thus, applying It\^{o}'s formula yields
\begin{eqnarray}
\|X_{t}^{\varepsilon}-\bar{X}_{t}\|_{H_1}^{2}=\!\!\!\!\!\!\!\!&&2\int_{0} ^{t}{}_{V_1^*}\langle A_1\left(X^{\vare}_{s},\mathscr{L}_{X^{\vare}_{s}}\right)-A_1\left(\bar{X}_{s},\mathscr{L}_{\bar{X}_{s}}\right), X_{s}^{\varepsilon}-\bar{X}_{s}\rangle_{V_1} ds
\nonumber \\
 \!\!\!\!\!\!\!\!&&+ 2\int_{0} ^{t}\langle f\left(X^{\varepsilon}_s,\mathscr{L}_{X^{\varepsilon}_s},Y^{\varepsilon}_t\right)-\bar{f}\left(\bar{X}_{s},\mathscr{L}_{\bar{X}_{s}}\right), X_{s}^{\varepsilon}-\bar{X}_{s}\rangle_{H_1} ds\nonumber \\
 \!\!\!\!\!\!\!\!&& +\int_{0} ^{t}\left\|B_1\left(X^{\vare}_{s},\mathscr{L}_{X^{\vare}_{s}}\right)-B_1\left(\bar{X}_{s},\mathscr{L}_{\bar{X}_{s}}\right)\right\|_{L_2(U_1,H_1)}^2ds
\nonumber \\
 \!\!\!\!\!\!\!\!&& +2\int_{0} ^{t}\left\langle  \left[B_1\left(X^{\vare}_{s},\mathscr{L}_{X^{\vare}_{s}}\right)-B_1\left(\bar{X}_{s},\mathscr{L}_{\bar{X}_{s}}\right)\right]dW^{1}_s,  X_{s}^{\varepsilon}-\bar{X}_{s} \right\rangle_{H_1} \nonumber\\
:=\!\!\!\!\!\!\!\!&&I_{1}(t)+I_{2}(t)+I_{3}(t)+I_{4}(t).  \label{p1}
\end{eqnarray}
By condition $({\mathbf{A}}{\mathbf{2}})$, we have
\begin{eqnarray}
\EE\left(\sup_{t\in[0,T]}(I_1(t)+I_3(t))\right)\leq \!\!\!\!\!\!\!\!&&
C\EE\int_0^T\|X_{t}^{\varepsilon}-\bar{X}_{t}\|_{H_1}^2+\mathbb{W}_{2,H_1}(\mathscr{L}_{X^{\vare}_{t}},\mathscr{L}_{\bar{X}_{t}})^2dt
\nonumber\\\leq\!\!\!\!\!\!\!\!&&
C\EE\int_0^T\|X_{t}^{\varepsilon}-\bar{X}_{t}\|_{H_1}^2dt.\label{p2}
\end{eqnarray}
Then by Burkholder-Davis-Gundy's inequality, condition $({\mathbf{A}}{\mathbf{2}})$, it holds that
\begin{eqnarray}
\EE\left(\sup_{t\in[0,T]}I_4(t)\right)\leq \!\!\!\!\!\!\!\!&&
C\EE\left[\int_0^T\|B_1\left(X^{\vare}_{t},\mathscr{L}_{X^{\vare}_{t}}\right)-B_1\left(\bar{X}_{t},\mathscr{L}_{\bar{X}_{t}}\right)\|_{L_2(U_1,H_1)}^2
\|X_{t}^{\varepsilon}-\bar{X}_{t}\|_{H_1}^2dt\right]^{1/2}
\nonumber \\\leq\!\!\!\!\!\!\!\!&&C\EE\left[\sup_{t\in[0,T]}\|X_{t}^{\varepsilon}-\bar{X}_{t}\|_{H_1}^2\int_0^T\|B_1\left(X^{\vare}_{t},\mathscr{L}_{X^{\vare}_{t}}\right)-B_1\left(\bar{X}_{t},\mathscr{L}_{\bar{X}_{t}}\right)\|_{L_2(U_1,H_1)}^2dt\right]^{1/2}
\nonumber\\\leq\!\!\!\!\!\!\!\!&&\frac{1}{2}\EE\left[\sup_{t\in[0,T]}\|X_{t}^{\varepsilon}-\bar{X}_{t}\|_{H_1}^2\right]
+C\EE\int_0^T\|X_{t}^{\varepsilon}-\bar{X}_{t}\|_{H_1}^2+\mathbb{W}_{2,H_1}(\mathscr{L}_{X^{\vare}_{t}},\mathscr{L}_{\bar{X}_{t}})^2dt
\nonumber\\\leq\!\!\!\!\!\!\!\!&&\frac{1}{2}\EE\left[\sup_{t\in[0,T]}\|X_{t}^{\varepsilon}-\bar{X}_{t}\|_{H_1}^2\right]
+C\EE\int_0^T\|X_{t}^{\varepsilon}-\bar{X}_{t}\|_{H_1}^2dt.\label{p4}
\end{eqnarray}
As for $I_2(t)$, we first rewrite it as
\begin{eqnarray}
I_2(t)=\!\!\!\!\!\!\!\!&&2\int_{0} ^{t}\langle f\left(X^{\varepsilon}_s,\mathscr{L}_{X^{\varepsilon}_s},Y^{\varepsilon}_s\right)
-f\left(X^{\varepsilon}_{s(\delta)},\mathscr{L}_{X^{\varepsilon}_{s(\delta)}},\hat{Y}^{\varepsilon}_s\right), X_{s}^{\varepsilon}-\bar{X}_{s}\rangle_{H_1} ds\nonumber \\
\nonumber \\
 \!\!\!\!\!\!\!\!&&+ 2\int_{0} ^{t}\langle f\left(X^{\varepsilon}_{s(\delta)},\mathscr{L}_{X^{\varepsilon}_{s(\delta)}},\hat{Y}^{\varepsilon}_s\right)-\bar{f}\left(X^{\varepsilon}_{s(\delta)},\mathscr{L}_{X^{\varepsilon}_{s(\delta)}}\right), X_{s}^{\varepsilon}-\bar{X}_{s}\rangle_{H_1} ds\nonumber \\
 \!\!\!\!\!\!\!\!&& + 2\int_{0} ^{t}\langle \bar{f}\left(X^{\varepsilon}_{s(\delta)},\mathscr{L}_{X^{\varepsilon}_{s(\delta)}}\right)-
 \bar{f}\left(X^{\varepsilon}_{s},\mathscr{L}_{X^{\varepsilon}_s}\right), X_{s}^{\varepsilon}-\bar{X}_{s}\rangle_{H_1} ds\nonumber \\
 \!\!\!\!\!\!\!\!&& + 2\int_{0} ^{t}{}\langle
 \bar{f}\left(X^{\varepsilon}_{s},\mathscr{L}_{X_s^{\varepsilon}}\right)
 -\bar{f}\left(\bar{X}_{s},\mathscr{L}_{\bar{X}_{s}}\right), X_{s}^{\varepsilon}-\bar{X}_{s}\rangle_{H_1} ds\nonumber \\
:=\!\!\!\!\!\!\!\!&&I_{21}(t)+I_{22}(t)+I_{23}(t)+I_{24}(t).  \label{p5}
\end{eqnarray}
According to $({\mathbf{A}}{\mathbf{2}})$, \eref{F3.7}, \eref{3.14}, and \eref{p0}, it is easy to see that
\begin{eqnarray}
\!\!\!\!\!\!\!\!&&\EE\left(\sup_{t\in[0,T]}(I_{21}(t)+I_{23}(t))\right)\nonumber \\\leq \!\!\!\!\!\!\!\!&&
C\EE\int_0^T\left(\|X_{t}^{\varepsilon}-X_{t(\delta)} ^{\varepsilon}\|_{H_1}+\mathbb{W}_{2,H_1}(\mathscr{L}_{X^{\vare}_{t}},\mathscr{L}_{X_{t(\delta)} ^{\varepsilon}})+\|Y_{t}^{\varepsilon}-\hat{Y}_{t} ^{\varepsilon}\|_{H_2}\right)\|X_{t}^{\varepsilon}-\bar{X}_{t}\|_{H_1}dt
\nonumber\\\leq\!\!\!\!\!\!\!\!&&
C\EE\int_0^T\|X_{t}^{\varepsilon}-\bar{X}_{t}\|_{H_1}^2dt
+C\EE\int_0^T\|X_{t}^{\varepsilon}-\bar{X}_{t(\delta)}\|_{H_1}^2dt
+C\EE\int_0^T\|Y_{t}^{\varepsilon}-\hat{Y}_{t} ^{\varepsilon}\|_{H_1}^2dt
\nonumber\\\leq\!\!\!\!\!\!\!\!&&
C_T\delta(1+\|x\|_{H_1}^2+\|y\|_{H_2}^2)+C\EE\int_0^T\|X_{t}^{\varepsilon}-\bar{X}_{t}\|_{H_1}^2dt
\label{p6}
\end{eqnarray}
and
\begin{eqnarray}
\EE\left(\sup_{t\in[0,T]}I_{24}(t)\right)\leq \!\!\!\!\!\!\!\!&&
C\EE\int_0^T\left(\|X_{t}^{\varepsilon}-\bar{X}_{t} ^{\varepsilon}\|_{H_1}+\mathbb{W}_{2,H_1}(\mathscr{L}_{X^{\vare}_{t}},\mathscr{L}_{\bar{X}_{t} ^{\varepsilon}})\right)\|X_{t}^{\varepsilon}-\bar{X}_{t}\|_{H_1}dt
\nonumber\\\leq\!\!\!\!\!\!\!\!&&
C\EE\int_0^T\|X_{t}^{\varepsilon}-\bar{X}_{t}\|_{H_1}^2dt.
\label{p7}
\end{eqnarray}
As for $I_{22}(t)$, we rewrite it as
\begin{eqnarray}
I_{22}(t)=\!\!\!\!\!\!\!\!&& 2\int_{0} ^{t}\langle f\left(X^{\varepsilon}_{s(\delta)},\mathscr{L}_{X^{\varepsilon}_{s(\delta)}},\hat{Y}^{\varepsilon}_s\right)-\bar{f}\left(X^{\varepsilon}_{s(\delta)},\mathscr{L}_{X^{\varepsilon}_{s(\delta)}}\right), X_{s}^{\varepsilon}-{X}^{\varepsilon}_{s(\delta)}\rangle_{H_1} ds\nonumber \\
 \!\!\!\!\!\!\!\!&& + 2\int_{0} ^{t}\langle f\left(X^{\varepsilon}_{s(\delta)},\mathscr{L}_{X^{\varepsilon}_{s(\delta)}},\hat{Y}^{\varepsilon}_s\right)-\bar{f}\left(X^{\varepsilon}_{s(\delta)},\mathscr{L}_{X^{\varepsilon}_{s(\delta)}}\right), X_{s(\delta)}^{\varepsilon}-{\bar{X}}_{s(\delta)}\rangle_{H_1} ds\nonumber \\
 \!\!\!\!\!\!\!\!&& + 2\int_{0} ^{t}\langle f\left(X^{\varepsilon}_{s(\delta)},\mathscr{L}_{X^{\varepsilon}_{s(\delta)}},\hat{Y}^{\varepsilon}_s\right)-\bar{f}\left(X^{\varepsilon}_{s(\delta)},\mathscr{L}_{X^{\varepsilon}_{s(\delta)}}\right), \bar{X}_{s(\delta)}-{\bar{X}}_{s}\rangle_{H_1} ds\nonumber \\
:=\!\!\!\!\!\!\!\!&&J_{1}(t)+J_{2}(t)+J_{3}(t).  \label{p8}
\end{eqnarray}
By Lemma \ref{PMY} and Lemma \ref{COX}, we obtain
\begin{eqnarray}
\!\!\!\!\!\!\!\!&&\EE\left(\sup_{t\in[0,T]}J_1(t)\right)\nonumber\\\leq \!\!\!\!\!\!\!\!&&
C\EE\int_0^T\|f(X^{\varepsilon}_{t(\delta)},\mathscr{L}_{X^{\varepsilon}_{t(\delta)}},\hat{Y}^{\varepsilon}_t)-\bar{f}(X^{\varepsilon}_{t(\delta)},\mathscr{L}_{X^{\varepsilon}_{t(\delta)}})\|_{H_1}\|X_{t}^{\varepsilon}-{X}_{t(\delta)}^{\varepsilon}\|_{H_1}dt
\nonumber\\\leq\!\!\!\!\!\!\!\!&&
C\left[\EE\int_0^T(1+\|X^{\varepsilon}_{t(\delta)}\|_{H_1}^2+\mathscr{L}_{X^{\varepsilon}_{t(\delta)}}(\|\cdot\|_{H_1}^2)+\|\hat{Y}^{\varepsilon}_{t}\|_{H_2}^2)dt\right]^{1/2}
\left[\EE\int_0^T\|X_{t}^{\varepsilon}-{X}_{t(\delta)}^{\varepsilon}\|_{H_1}^2dt\right]^{1/2}
\nonumber\\\leq\!\!\!\!\!\!\!\!&&
C_T\delta^{1/2}(1+\|x\|_{H_1}^2+\|y\|_{H_2}^2).\label{p9}
\end{eqnarray}
Similarly, by Lemma \ref{L3.8} we can also get
\begin{eqnarray}
\EE\left(\sup_{t\in[0,T]}J_3(t)\right)\leq
C_T\delta^{1/2}(1+\|x\|_{H_1}^2+\|y\|_{H_2}^2).\label{p10}
\end{eqnarray}
Thus, combining \eref{p1}-\eref{p10} yields
\begin{eqnarray}
\EE\left[\sup_{t\in[0,T]}\|X_{t}^{\varepsilon}-\bar{X}_{t}\|_{H_1}^2\right]\leq \!\!\!\!\!\!\!\!&& 2\EE\left(\sup_{t\in[0,T]}J_2(t)\right)+ C_T\delta^{1/2}(1+\|x\|_{H_1}^2+\|y\|_{H_2}^2)
\nonumber \\
 \!\!\!\!\!\!\!\!&& +C\EE\int_0^T\|X_{t}^{\varepsilon}-\bar{X}_{t}\|_{H_1}^2dt.\label{p11}
\end{eqnarray}

\textbf{Step 2}. In this step, we will use the time discretization technique to deal with $J_2(t)$. Note that
\begin{eqnarray}
 |J_2(t)|=\!\!\!\!\!\!\!\!&&2\left|\sum_{k=0}^{[t/\delta]-1}\int_{k\delta} ^{(k+1)\delta}\left\langle f\left(X^{\varepsilon}_{s(\delta)},\mathscr{L}_{X^{\varepsilon}_{s(\delta)}},\hat{Y}^{\varepsilon}_s\right)-\bar{f}\left(X^{\varepsilon}_{s(\delta)},\mathscr{L}_{X^{\varepsilon}_{s(\delta)}}\right), X_{s(\delta)}^{\varepsilon}-{\bar{X}}_{s(\delta)}\right\rangle_{H_1} ds\right.\nonumber \\
 \!\!\!\!\!\!\!\!&& +\left.\int_{t(\delta)} ^{t}\left\langle f\left(X^{\varepsilon}_{s(\delta)},\mathscr{L}_{X^{\varepsilon}_{s(\delta)}},\hat{Y}^{\varepsilon}_s\right)-\bar{f}\left(X^{\varepsilon}_{s(\delta)},\mathscr{L}_{X^{\varepsilon}_{s(\delta)}}\right), X_{s(\delta)}^{\varepsilon}-{\bar{X}}_{s(\delta)}\right\rangle_{H_1} ds\right|\nonumber \\
 \leq\!\!\!\!\!\!\!\!&&2\sum_{k=0}^{[t/\delta]-1}\left|\int_{k\delta} ^{(k+1)\delta}\left\langle f\left(X^{\varepsilon}_{s(\delta)},\mathscr{L}_{X^{\varepsilon}_{s(\delta)}},\hat{Y}^{\varepsilon}_s\right)-\bar{f}\left(X^{\varepsilon}_{s(\delta)},\mathscr{L}_{X^{\varepsilon}_{s(\delta)}}\right), X_{s(\delta)}^{\varepsilon}-{\bar{X}}_{s(\delta)}\right\rangle_{H_1} ds\right|\nonumber \\
 \!\!\!\!\!\!\!\!&& +2\left|\int_{t(\delta)} ^{t}\left\langle f\left(X^{\varepsilon}_{s(\delta)},\mathscr{L}_{X^{\varepsilon}_{s(\delta)}},\hat{Y}^{\varepsilon}_s\right)-\bar{f}\left(X^{\varepsilon}_{s(\delta)},\mathscr{L}_{X^{\varepsilon}_{s(\delta)}}\right), X_{s(\delta)}^{\varepsilon}-{\bar{X}}_{s(\delta)}\right\rangle_{H_1} ds\right|\nonumber \\
:=\!\!\!\!\!\!\!\!&&J_{21}(t)+J_{22}(t).  \label{p12}
\end{eqnarray}
By Lemma \ref{PMY} and Lemma \ref{COX}, it is easy to prove that
\begin{eqnarray}
\!\!\!\!\!\!\!\!&&\EE\left(\sup_{t\in[0,T]}J_{22}(t)\right)\nonumber\\\leq \!\!\!\!\!\!\!\!&&
C\left[\EE\sup_{t\in[0,T]}\int_{t(\delta)} ^{t}\left(1+\|X^{\varepsilon}_{s(\delta)}\|_{H_1}^2+\mathscr{L}_{X^{\varepsilon}_{s(\delta)}}(\|\cdot\|_{H_1}^2)+\|\hat{Y}^{\varepsilon}_{s}\|_{H_2}^2\right)ds\right]^{1/2}
\nonumber\\ \!\!\!\!\!\!\!\!&&\times\left[\EE\sup_{t\in[0,T]}\int_{t(\delta)}^{t}\|X_{s(\delta)}^{\varepsilon}-\bar{X}_{s(\delta)}\|_{H_1}^2ds\right]^{1/2}
\nonumber\\\leq\!\!\!\!\!\!\!\!&&
C\delta^{1/2}\left[\EE\int_{0}^{T}\left(1+\|X^{\varepsilon}_{t(\delta)}\|_{H_1}^2+\mathscr{L}_{X^{\varepsilon}_{t(\delta)}}(\|\cdot\|_{H_1}^2)+\|\hat{Y}^{\varepsilon}_{t}\|_{H_2}^2\right)dt\right]^{1/2}
\left[\EE\sup_{t\in[0,T]}\|X_{t}^{\varepsilon}-\bar{X}_{t}\|_{H_1}^2\right]^{1/2}
\nonumber\\\leq\!\!\!\!\!\!\!\!&&
C_T\delta^{1/2}(1+\|x\|_{H_1}^2+\|y\|_{H_2}^2).\label{p13}
\end{eqnarray}
As for the term $J_{21}(t)$, we can control it as follows.
\begin{eqnarray}
\!\!\!\!\!\!\!\!&&\EE\left(\sup_{t\in[0,T]}J_{21}(t)\right)\nonumber\\
\leq\!\!\!\!\!\!\!\!&&2\EE\sum_{k=0}^{[T/\delta]-1}\left|\int_{k\delta} ^{(k+1)\delta}\left\langle f\left(X^{\varepsilon}_{k\delta},\mathscr{L}_{X^{\varepsilon}_{k\delta}},\hat{Y}^{\varepsilon}_s\right)-\bar{f}\left(X^{\varepsilon}_{k\delta},\mathscr{L}_{X^{\varepsilon}_{k\delta}}\right), X_{k\delta}^{\varepsilon}-{\bar{X}}_{k\delta}\right\rangle_{H_1} ds\right|\nonumber \\
\leq\!\!\!\!\!\!\!\!&&\frac{C_T}{\delta}\max_{0\leq k\leq[T/\delta]-1}\EE\left|\int_{k\delta} ^{(k+1)\delta}\left\langle f\left(X^{\varepsilon}_{k\delta},\mathscr{L}_{X^{\varepsilon}_{k\delta}},\hat{Y}^{\varepsilon}_s\right)-\bar{f}\left(X^{\varepsilon}_{k\delta},\mathscr{L}_{X^{\varepsilon}_{k\delta}}\right), X_{k\delta}^{\varepsilon}-{\bar{X}}_{k\delta}\right\rangle_{H_1} ds\right|\nonumber \\
\leq\!\!\!\!\!\!\!\!&&\frac{C_T\varepsilon}{\delta}\max_{0\leq k\leq[T/\delta]-1}\left[\EE\left\|\int_{0} ^{\frac{\delta}{\varepsilon}} f\left(X^{\varepsilon}_{k\delta},\mathscr{L}_{X^{\varepsilon}_{k\delta}},\hat{Y}^{\varepsilon}_{s\varepsilon+k\delta}\right)-\bar{f}\left(X^{\varepsilon}_{k\delta},\mathscr{L}_{X^{\varepsilon}_{k\delta}}\right)ds\right\|_{H_1}^2 \right]^{1/2}
\nonumber \\
\!\!\!\!\!\!\!\!&&\cdot\left[\sup_{t\in[0,T]}\EE\|X_{t}^{\varepsilon}-{\bar{X}}_{t}\|_{H_1}^2\right]^{1/2}\nonumber \\
\leq\!\!\!\!\!\!\!\!&&\frac{C_T\varepsilon^2}{\delta^2}\max_{0\leq k\leq[T/\delta]-1}\left[\EE\left\|\int_{0} ^{\frac{\delta}{\varepsilon}} f\left(X^{\varepsilon}_{k\delta},\mathscr{L}_{X^{\varepsilon}_{k\delta}},\hat{Y}^{\varepsilon}_{s\varepsilon+k\delta}\right)-\bar{f}\left(X^{\varepsilon}_{k\delta},\mathscr{L}_{X^{\varepsilon}_{k\delta}}\right)ds\right\|_{H_1}^2 \right]
\nonumber \\
\!\!\!\!\!\!\!\!&&+\frac{1}{4}\left[\sup_{t\in[0,T]}\EE\|X_{t}^{\varepsilon}-{\bar{X}}_{t}\|_{H_1}^2\right]
\nonumber \\
\leq\!\!\!\!\!\!\!\!&&
\frac{C_T\varepsilon^2}{\delta^2}\max_{0\leq k\leq[T/\delta]-1}\left[\int_{0} ^{\frac{\delta}{\varepsilon}} \int_{r} ^{\frac{\delta}{\varepsilon}}\Phi_k(s,r)dsdr \right]+\frac{1}{4}\EE\left[\sup_{t\in[0,T]}\|X_{t}^{\varepsilon}-{\bar{X}}_{t}\|_{H_1}^2\right],\label{w1}
\end{eqnarray}
where for any $0\leq r\leq s\leq \frac{\delta}{\varepsilon}$,
\begin{eqnarray*}
\Phi_k(s,r):=\!\!\!\!\!\!\!\!&&\EE\left[\left\langle f\left(X^{\varepsilon}_{k\delta},\mathscr{L}_{X^{\varepsilon}_{k\delta}},\hat{Y}^{\varepsilon}_{s\varepsilon+k\delta}\right)-\bar{f}\left(X^{\varepsilon}_{k\delta},\mathscr{L}_{X^{\varepsilon}_{k\delta}}\right),
\right.\right.\nonumber \\
\!\!\!\!\!\!\!\!&&\left.\left.~~~~~~f\left(X^{\varepsilon}_{k\delta},\mathscr{L}_{X^{\varepsilon}_{k\delta}},\hat{Y}^{\varepsilon}_{r\varepsilon+k\delta}\right)-\bar{f}\left(X^{\varepsilon}_{k\delta},\mathscr{L}_{X^{\varepsilon}_{k\delta}}\right)\right\rangle_{H_1}\right].
\end{eqnarray*}

For any $s>0$, $\mu\in\mathscr{P}_2(H_1)$, and any $\mathscr{F}_s$-measurable $H_1$-valued random variable $X$ and $H_2$-valued random variable $Y$, we consider the following equation
\begin{eqnarray}\label{p14}
\left\{ \begin{aligned}
&d\tilde{Y}_{t}=\frac{1}{\varepsilon}[A_2(X,\mu,\tilde{Y}_{t})]dt+\frac{1}{\sqrt{\varepsilon}}B_2(X,\mu,\tilde{Y}_t)d{{W}}_{t}^{2},~~t\geq s,\\
&\tilde{Y}_{s}=Y.
\end{aligned} \right.
\end{eqnarray}
 Then, by \cite[Theorem 4.2.4]{LR1}, it is easy to see that Eq.~\eref{p14} has a unique solution denoted by $\tilde{Y}_t^{\varepsilon,s,X,\mu,Y}$.
 Following the construction of $\hat{Y}_t^\varepsilon$ in \eref{4.6a}, for any $k\in \mathbb{N}$, it is easy to check that
 $$\hat{Y}_t^\varepsilon=\tilde{Y}_t^{\varepsilon,k\delta,X^\varepsilon_{k\delta},
 \mathscr{L}_{X^{\varepsilon}_{k\delta}},\hat{Y}^\varepsilon_{k\delta}},~~
 t\in[k\delta,(k+1)\delta].$$
Thus, we have
\begin{eqnarray*}
\Phi_k(s,r)=\!\!\!\!\!\!\!\!&&\EE\left[\left\langle f\left(X^{\varepsilon}_{k\delta},\mathscr{L}_{X^{\varepsilon}_{k\delta}},\tilde{Y}^{\varepsilon,k\delta,X_{k\delta}^\varepsilon,\mathscr{L}_{X^{\varepsilon}_{k\delta}},\hat{Y}^\varepsilon_{k\delta}}_{s\varepsilon+k\delta}\right)
-\bar{f}\left(X^{\varepsilon}_{k\delta},\mathscr{L}_{X^{\varepsilon}_{k\delta}}\right),
\right.\right.\nonumber \\
\!\!\!\!\!\!\!\!&&\left.\left.~~~~~~f\left(X^{\varepsilon}_{k\delta},\mathscr{L}_{X^{\varepsilon}_{k\delta}},\tilde{Y}^{\varepsilon,k\delta,X_{k\delta}^\varepsilon,\mathscr{L}_{X^{\varepsilon}_{k\delta}},\hat{Y}^\varepsilon_{k\delta}}_{r\varepsilon+k\delta}\right)-\bar{f}\left(X^{\varepsilon}_{k\delta},\mathscr{L}_{X^{\varepsilon}_{k\delta}}\right)\right\rangle_{H_1}\right].
\end{eqnarray*}
Note that for any fixed $x\in H_1$ and
$y\in H_2$, $\tilde{Y}_{s\varepsilon+k\delta}^{\varepsilon,k\delta,x,\mu,y}$ is independent of $\mathscr{F}_{k\delta}$, and $X_{k\delta}^\varepsilon$, $\hat{Y}_{k\delta}^\varepsilon$ are $\mathscr{F}_{k\delta}$-measurable, thus we have
\begin{eqnarray*}
\Phi_k(s,r)=\!\!\!\!\!\!\!\!&&\EE\Big\{\EE\Big[\langle f\big(X^{\varepsilon}_{k\delta},\mathscr{L}_{X^{\varepsilon}_{k\delta}},\tilde{Y}^{\varepsilon,k\delta,X_{k\delta}^\varepsilon,\mathscr{L}_{X^{\varepsilon}_{k\delta}},\hat{Y}^\varepsilon_{k\delta}}_{s\varepsilon+k\delta}\big)-\bar{f}\big(X^{\varepsilon}_{k\delta},\mathscr{L}_{X^{\varepsilon}_{k\delta}}\big),
\nonumber\\
\!\!\!\!\!\!\!\!&&~~~~~~f\big(X^{\varepsilon}_{k\delta},\mathscr{L}_{X^{\varepsilon}_{k\delta}},\tilde{Y}^{\varepsilon,k\delta,X_{k\delta}^\varepsilon,\mathscr{L}_{X^{\varepsilon}_{k\delta}},\hat{Y}^\varepsilon_{k\delta}}_{r\varepsilon+k\delta}\big)-\bar{f}\big(X^{\varepsilon}_{k\delta},
\mathscr{L}_{X^{\varepsilon}_{k\delta}}\big)\rangle_{H_1}\big|\mathscr{F}_{k\delta}\Big] \Big\}
\nonumber\\
=\!\!\!\!\!\!\!\!&&
\EE\Big\{\EE\Big[\langle f\big(x,\mathscr{L}_{X^{\varepsilon}_{k\delta}},\tilde{Y}^{\varepsilon,k\delta,x,\mathscr{L}_{X^{\varepsilon}_{k\delta}},y}_{s\varepsilon+k\delta}\big)-\bar{f}\big(x,\mathscr{L}_{X^{\varepsilon}_{k\delta}}\big),
\nonumber\\
\!\!\!\!\!\!\!\!&&~~~~~~f\big(x,\mathscr{L}_{X^{\varepsilon}_{k\delta}},\tilde{Y}^{\varepsilon,k\delta,x,\mathscr{L}_{X^{\varepsilon}_{k\delta}},y}_{r\varepsilon+k\delta}\big)-\bar{f}\big(x,
\mathscr{L}_{X^{\varepsilon}_{k\delta}}\big)\rangle_{H_1}\Big]\Big|_{(x,y)=(X^{\varepsilon}_{k\delta},\hat{Y}^\varepsilon_{k\delta})}\Big\}.
\end{eqnarray*}
Recall the definition of the process $\{\tilde{Y}_{s\varepsilon+k\delta}^{\varepsilon,k\delta,x,\mu,y}\}_{s\geq0}$, it is easy to see that
\begin{eqnarray}
\tilde{Y}_{s\varepsilon+k\delta}^{\varepsilon,k\delta,x,\mu,y}
=\!\!\!\!\!\!\!\!&&y+\frac{1}{\varepsilon}\int_{k\delta}^{s\varepsilon+k\delta}A_2(x,\mu,\tilde{Y}_{r}^{\varepsilon,k\delta,x,\mu,y})dr
+\frac{1}{\sqrt{\varepsilon}}\int_{k\delta}^{s\varepsilon+k\delta}B_2(x,\mu,\tilde{Y}_{r}^{\varepsilon,k\delta,x,\mu,y})d{{W}}_{r}^{2}
\nonumber \\
=\!\!\!\!\!\!\!\!&&
y+\frac{1}{\varepsilon}\int_{0}^{s\varepsilon}A_2(x,\mu,\tilde{Y}_{r+k\delta}^{\varepsilon,k\delta,x,\mu,y})dr
+\frac{1}{\sqrt{\varepsilon}}\int_{0}^{s\varepsilon}B_2(x,\mu,\tilde{Y}_{r+k\delta}^{\varepsilon,k\delta,x,\mu,y})d{{W}}_{r}^{2,k\delta}
\nonumber \\
=\!\!\!\!\!\!\!\!&&
y+\int_{0}^{s}A_2(x,\mu,\tilde{Y}_{r\varepsilon+k\delta}^{\varepsilon,k\delta,x,\mu,y})dr
+\int_{0}^{s}B_2(x,\mu,\tilde{Y}_{r\varepsilon
+k\delta}^{\varepsilon,k\delta,x,\mu,y})d{\hat{W}}_{r}^{2,k\delta},\label{p20}
\end{eqnarray}
where $$\Big\{W_{r}^{2,k\delta}:=W_{r+k\delta}^{2}-W_{k\delta}^{2}\Big\}_{r\geq0}~~\text{and}~~\Big\{\hat{W}_{r}^{2,k\delta}:=\frac{1}{\sqrt{\varepsilon}}W_{r\varepsilon}^{2,k\delta}\Big\}_{r\geq0}.$$

Note that the solution of the frozen equation satisfies
\begin{eqnarray}
{Y}_{s}^{x,\mu,y}
=
y+\int_{0}^{s}A_2(x,\mu,{Y}_{r}^{x,\mu,y})dr
+\int_{0}^{s}B_2(x,\mu,{Y}_{r}^{x,\mu,y})d\tilde{W}_{r}^{2}.\label{p21}
\end{eqnarray}
Then, the uniqueness of the solution of \eref{p20} and \eref{p21} implies that the distribution of $\left\{\tilde{Y}_{s\varepsilon+k\delta}^{\varepsilon,k\delta,x,\mu,y}\right\}_{0\leq s\leq\frac{\delta}{\varepsilon}}$ coincides with the distribution $\left\{{Y}_{s}^{x,\mu,y}\right\}_{0\leq s\leq\frac{\delta}{\varepsilon}}.$ Thus, using Markov and time-homogenous properties of process ${Y}_{s}^{x,\mu,y}$, we have

\begin{eqnarray}
\!\!\!\!\!\!\!\!&&\Phi_k(s,r)
\nonumber \\=\!\!\!\!\!\!\!\!&&
\EE\Big\{\tilde{\EE}\Big[\Big\langle f\Big(x,\mathscr{L}_{X^{\varepsilon}_{k\delta}},{Y}^{x,\mathscr{L}_{X^{\varepsilon}_{k\delta}},y}_{s}\Big)-\bar{f}\Big(x,\mathscr{L}_{X^{\varepsilon}_{k\delta}}\Big),
\nonumber \\
\!\!\!\!\!\!\!\!&&~~~~~~
f\Big(x,\mathscr{L}_{X^{\varepsilon}_{k\delta}},
{Y}^{x,\mathscr{L}_{X^{\varepsilon}_{k\delta}},y}_{r}\Big)-\bar{f}\Big(x,
\mathscr{L}_{X^{\varepsilon}_{k\delta}}\Big)\Big\rangle_{H_1}\Big]\Big|_{(x,y)=(X^{\varepsilon}_{k\delta},\hat{Y}^\varepsilon_{k\delta})}\Big\}
\nonumber \\
=\!\!\!\!\!\!\!\!&&
\EE\Big\{\tilde{\EE}\Big[\Big\langle \tilde{\EE}\Big[f\Big(x,\mathscr{L}_{X^{\varepsilon}_{k\delta}},{Y}^{x,\mathscr{L}_{X^{\varepsilon}_{k\delta}},y}_{s}\Big)-\bar{f}\Big(x,\mathscr{L}_{X^{\varepsilon}_{k\delta}}\Big)\big|\tilde{\mathscr{F}}_{r}\Big],
\nonumber \\
\!\!\!\!\!\!\!\!&&~~~~~~f\Big(x,\mathscr{L}_{X^{\varepsilon}_{k\delta}},
{Y}^{x,\mathscr{L}_{X^{\varepsilon}_{k\delta}},y}_{r}\Big)-\bar{f}\Big(x,
\mathscr{L}_{X^{\varepsilon}_{k\delta}}\Big)\Big\rangle_{H_1}\Big]\Big|_{(x,y)=(X^{\varepsilon}_{k\delta},\hat{Y}^\varepsilon_{k\delta})}\Big\}
\nonumber \\
=\!\!\!\!\!\!\!\!&&
\EE\Big\{\tilde{\EE}\Big[\Big\langle \tilde{\EE}\Big[f\Big(x,\mathscr{L}_{X^{\varepsilon}_{k\delta}},{Y}^{x,\mathscr{L}_{X^{\varepsilon}_{k\delta}},z}_{s-r}\Big)-\bar{f}\Big(x,\mathscr{L}_{X^{\varepsilon}_{k\delta}}\Big)\Big]\mathbf{1}_{\{z=Y^{x,\mathscr{L}_{X^{\varepsilon}_{k\delta}},y}_r\}},
\nonumber \\
\!\!\!\!\!\!\!\!&&~~~~~~f\Big(x,\mathscr{L}_{X^{\varepsilon}_{k\delta}},
{Y}^{x,\mathscr{L}_{X^{\varepsilon}_{k\delta}},y}_{r}\Big)-\bar{f}\Big(x,
\mathscr{L}_{X^{\varepsilon}_{k\delta}}\Big)\Big\rangle_{H_1}\Big]\Big|_{(x,y)=(X^{\varepsilon}_{k\delta},\hat{Y}^\varepsilon_{k\delta})}\Big\}.
\nonumber
\end{eqnarray}

Therefore, according to {Proposition} \ref{Rem 4.0} and \ref{Rem 4.1}, we arrive
\begin{eqnarray}
\Phi_k(s,r)
\leq\!\!\!\!\!\!\!\!&&
C_T\EE\left\{\tilde{\EE}\left[1+\|X^{\varepsilon}_{k\delta}\|_{H_1}^2+\mathscr{L}_{X^{\varepsilon}_{k\delta}}(\|\cdot\|_{H_1}^2)+\|{Y}^{X_{k\delta}^\varepsilon,\mathscr{L}_{X^{\varepsilon}_{k\delta}},\hat{Y}^\varepsilon_{k\delta}}_{r}\|_{H_2}^2\right]e^{-\frac{(s-r)\rho}{2}}\right\}
\nonumber \\\leq\!\!\!\!\!\!\!\!&&
C_T\EE\left(1+\|X^{\varepsilon}_{k\delta}\|_{H_1}^2+\mathscr{L}_{X^{\varepsilon}_{k\delta}}(\|\cdot\|_{H_1}^2)+\|\hat{{Y}}^{\varepsilon}_{k\delta}\|_{H_2}^2\right)e^{-\frac{(s-r)\rho}{2}}
\nonumber \\\leq\!\!\!\!\!\!\!\!&& C_T(1+\|x\|_{H_1}^2+\|y\|_{H_2}^2)e^{-\frac{(s-r)\rho}{2}}.\label{w2}
\end{eqnarray}
By \eref{w1} and \eref{w2}, we deduce that
\begin{eqnarray}
\EE\left(\sup_{t\in[0,T]}J_{21}(t)\right)
\leq\!\!\!\!\!\!\!\!&&
C_T(1+\|x\|_{H_1}^2+\|y\|_{H_2}^2)\frac{\varepsilon^2}{\delta^2}\left[\int_{0} ^{\frac{\delta}{\varepsilon}} \int_{r} ^{\frac{\delta}{\varepsilon}}e^{-\frac{(s-r)\rho}{2}}dsdr \right]
\nonumber \\
\!\!\!\!\!\!\!\!&&
+\frac{1}{4}\EE\left[\sup_{t\in[0,T]}\|X_{t}^{\varepsilon}-{\bar{X}}_{t}\|_{H_1}^2\right]
\nonumber \\
=\!\!\!\!\!\!\!\!&&
C_T(1+\|x\|_{H_1}^2+\|y\|_{H_2}^2)\frac{\varepsilon^2}{\delta^2}
\left(\frac{2\delta}{\rho\varepsilon}-\frac{4}{\rho^2}+\frac{4}{\rho^2}e^{\frac{-\rho\delta}{\varepsilon}}\right)
\nonumber \\
\!\!\!\!\!\!\!\!&&
+\frac{1}{4}\EE\left[\sup_{t\in[0,T]}\|X_{t}^{\varepsilon}-{\bar{X}}_{t}\|_{H_1}^2\right]
\nonumber \\\leq\!\!\!\!\!\!\!\!&&
C_T(1+\|x\|_{H_1}^2+\|y\|_{H_2}^2)
\left(\frac{\varepsilon^2}{\delta^2}+\frac{\varepsilon}{\delta}\right)+\frac{1}{4}\EE\left[\sup_{t\in[0,T]}\|X_{t}^{\varepsilon}-{\bar{X}}_{t}\|_{H_1}^2\right].~~~\label{w3}
\end{eqnarray}

\textbf{Step 3}. Now, we are in the position to complete the proof. Combining \eref{p11}-\eref{p13} and \eref{w3} yields
\begin{eqnarray}
\EE\left[\sup_{t\in[0,T]}\|X_{t}^{\varepsilon}-\bar{X}_{t}\|_{H_1}^2\right]\leq \!\!\!\!\!\!\!\!&& C_T(1+\|x\|_{H_1}^2+\|y\|_{H_2}^2)
\left(\frac{\varepsilon^2}{\delta^2}+\frac{\varepsilon}{\delta}+\delta^{1/2}\right)
\nonumber \\
 \!\!\!\!\!\!\!\!&& +C\EE\int_0^T\|X_{t}^{\varepsilon}-\bar{X}_{t}\|_{H_1}^2dt.\nonumber
\end{eqnarray}
Using the Gronwall's inequality yields
\begin{eqnarray}
\EE\left[\sup_{t\in[0,T]}\|X_{t}^{\varepsilon}-\bar{X}_{t}\|_{H_1}^2\right]\leq \!\!\!\!\!\!\!\!&& C_T(1+\|x\|_{H_1}^2+\|y\|_{H_2}^2)
\left(\frac{\varepsilon^2}{\delta^2}+\frac{\varepsilon}{\delta}+\delta^{1/2}\right)
.~~\label{w4}
\end{eqnarray}
Then, by taking $\delta=\varepsilon^{2/3}$ in \eref{w4} we deduce that
\begin{eqnarray}
\EE\left[\sup_{t\in[0,T]}\|X_{t}^{\varepsilon}-\bar{X}_{t}\|_{H_1}^2\right]\leq \!\!\!\!\!\!\!\!&& C_T(1+\|x\|_{H_1}^2+\|y\|_{H_2}^2)\varepsilon^{1/3}.\nonumber
\end{eqnarray}
The proof is complete. \hspace{\fill}$\Box$

\section{Application to examples}\label{sec5}
\setcounter{equation}{0}
 \setcounter{definition}{0}
 In this section, we shall apply the main results  in Theorem \ref{Th1} and \ref{main result 1} to various two-time-scale Mckean-Vlasov SPDE models, which also generalize some existing works in the literature from  classical SPDEs to distribution dependent case.

Throughout this section, we assume $\Lambda\subset\mathbb{R}^d$ as a bounded domain with smooth boundary $\partial\Lambda$. Let
$C_0^\infty(\Lambda, \mathbb{R}^d)$ be the space of all smooth functions from $\Lambda$ to $\mathbb{R}^d$ with compact support.  For any $r\ge 1$, let $L^r(\Lambda, \mathbb{R}^d)$ be the vector valued $L^r$-space with the norm $\|\cdot\|_{L^r}$.
For any integer $m>0$, we denote by $W_0^{m,r}(\Lambda, \mathbb{R}^d)$ the classical Sobolev space (with Dirichlet boundary condition) from domain $\Lambda$
to $\mathbb{R}^d$ equipped with the (equivalent) norm
$$ \|u\|_{W^{m,r}} = \left( \sum_{ |\alpha|= m} \int_{\Lambda} |D^\alpha u|^rd x \right)^\frac{1}{r}.$$

\subsection{Slow-fast Mckean-Vlasov stochastic porous media equation}
The first example is the Mckean-Vlasov stochastic porous media type equation, which is the dynamics of gas flow in a porous medium (cf.~e.g.~\cite{BDR2,DGT,Gess1,V07}). More precisely, we consider the following slow-fast McKean-Vlasov stochastic evolution equations
\begin{equation}\label{ex1}
\left\{ \begin{aligned}
&dX^{\varepsilon}_t=\left[\Delta\Psi(X^{\varepsilon}_t,\mathscr{L}_{X^{\varepsilon}_t})
+f(X^{\varepsilon}_t,\mathscr{L}_{X^{\varepsilon}_t},Y^{\varepsilon}_t)\right]dt
+B_1(X^{\varepsilon}_t,\mathscr{L}_{X^{\varepsilon}_t})dW^1_t,\\
&dY^{\varepsilon}_t=\frac{1}{\varepsilon}[\Delta Y^{\varepsilon}_t
+g(X^{\varepsilon}_t,\mathscr{L}_{X^{\varepsilon}_t},Y^{\varepsilon}_t)]dt+\frac{1}{\sqrt{\varepsilon}}B_2(X^{\varepsilon}_t,\mathscr{L}_{X^{\varepsilon}_t},Y^{\varepsilon}_t)d W^{2}_{t},
\\&X^{\varepsilon}_0=x, Y^{\varepsilon}_0=y,
\end{aligned} \right.
\end{equation}
where $\Delta$ denotes the Laplace operator, and $\Psi,f,g,B_1,B_2$ satisfy some assumptions below.

For any $r\geq 2$, we set the following Gelfand triple for the slow equation
$$V_1:=L^r(\Lambda)\subset  H_1:= (W_0^{1,2}(\Lambda))^*  \subset  V_1^*,$$
and the following Gelfand triple for the fast equation
$$V_2:=W_0^{1,2}(\Lambda)\subset  H_2:= L^{2}(\Lambda)  \subset  V_2^*.$$

We recall the following useful lemma (see e.g.\cite[Lemma 4.1.13]{LR1}).
\begin{lemma}\label{l7}
The map
$$\Delta: W_0^{1,2}(\Lambda)\to (L^r(\Lambda))^*$$
could be extend to a linear isometry
$$\Delta:L^{\frac{r}{r-1}}(\Lambda)\to (L^r(\Lambda))^*.$$
Furthermore, for any $u\in L^{\frac{r}{r-1}}(\Lambda)$, $v\in L^r(\Lambda)$ we have
$$_{{V_1}^*}\langle-\Delta u,v\rangle_{V_1}=_{L^{\frac{r}{r-1}}}\langle u,v\rangle_{L^r}=\int_\Lambda u(\xi)v(\xi)d\xi.$$
\end{lemma}

We first formulate the assumptions on $\Psi$. Suppose the map
$$\Psi:V_1\times\mathscr{P}_2(H_1)\to L^{\frac{r}{r-1}}(\Lambda)$$
is measurable, and satisfies the following hypothesis.
\begin{hypothesis}\label{h3}
For all $u,v\in V_1$ and $\mu,\nu\in\mathscr{P}_2(H_1)$,
\begin{enumerate}
\item [$(\Psi1)$] The map
\begin{eqnarray*}
V_1\times\mathscr{P}_2(H_1)\ni(u,\mu)\mapsto\int_{\Lambda}\Psi(u,\mu)(\xi)v(\xi)d\xi
\end{eqnarray*}
is continuous.
\item [$(\Psi2)$] There are some constants $C,\theta>0$ such that
\begin{eqnarray*}
\int_{\Lambda}\Psi(u,\mu)(\xi)v(\xi)d\xi\geq -C\big(1+\|u\|_{H_1}^2+\mu(\|\cdot\|_{H_1}^2)\big)+\theta\|u\|_{V_1}^r.
\end{eqnarray*}
\item [$(\Psi3)$]
\begin{eqnarray*}
\int_{\Lambda}\big(\Psi(u,\mu)(\xi)-\Psi(v,\nu)(\xi)\big)\big(u(\xi)-v(\xi)\big)d\xi\geq 0.
\end{eqnarray*}
\item [$(\Psi4)$] There is a constant $C>0$,
\begin{eqnarray*}
\|\Psi(u,\mu)\|_{L^{\frac{r}{r-1}}}^{\frac{r}{r-1}}\leq C\big(1+\|u\|_{V_1}^{r}+\mu(\|\cdot\|_{H_1}^2)\big).
\end{eqnarray*}
\end{enumerate}
\end{hypothesis}
After the preparations above, we now define map $A_1: V_1\times\mathscr{P}_2(H_1)\to {V_1}^*$ by
$$A_1(u,\mu):=\Delta\Psi(u,\mu).$$
The Lemma \ref{l7} ensures that the map $A_1$ is well-defined and takes value in ${V_1}^*$.  Moreover, it is easy to check that the conditions $(\Psi1)$-$(\Psi4)$ imply
$({\mathbf{A}}{\mathbf{1}})$-$({\mathbf{A}}{\mathbf{4}})$. In order to prove the main result, we further assume that the measurable maps
$$f:H_1\times\mathscr{P}_2(H_1)\times H_2\to H_1,~B_1:V_1\times\mathscr{P}_2(H_1)\to L_2(U_1,H_1),$$
and $$g:H_1\times\mathscr{P}_2(H_1)\times V_2\to V_2^*,~B_2:H_1\times\mathscr{P}_2(H_1)\times V_2\to L_2(U_2,H_2)$$
are Lipschitz continuous. More precisely, there are some positive constants $L_g$, $L_{B_2}$ and $C$ such that for all $u_1,u_2\in H_1$,$v_1,v_2\in H_2$ and $\mu_1,\mu_2\in\mathscr{P}_2(H_1)$,
\begin{eqnarray}
\!\!\!\!\!\!\!\!&&\|f(u_1,\mu_1,v_1)-f(u_2,\mu_2,v_2)\|_{H_1}\leq C\big(\|u_1-u_2\|_{H_1}+\|v_1-v_2\|_{H_2}+\mathbb{W}_{2,H_1}(\mu_1,\mu_2)\big),\label{55}
\\
\!\!\!\!\!\!\!\!&&\|B_1(u_1,\mu_1)-B_1(u_2,\mu_2)\|_{L_2(U_1,H_1)}
\leq C(\|u_1-u_2\|_{H_1}+\mathbb{W}_{2,H_1}(\mu_1,\mu_2)),\label{56}
\\
\!\!\!\!\!\!\!\!&&\|g(u_1,\mu_1,v_1)-g(u_2,\mu_2,v_2)\|_{H_1}\leq L_g\|v_1-v_2\|_{H_2}
\nonumber\\\!\!\!\!\!\!\!\!&&~~~~~~~~~~~~~~~~~~~~~~~~~~~~~~~~~~~~~~~~~~~~
+C\big(\|u_1-u_2\|_{H_1}+\mathbb{W}_{2,H_1}(\mu_1,\mu_2)\big),~\label{57}
\\
\!\!\!\!\!\!\!\!&&\|B_2(u_1,\mu_1,v_1)-B_2(u_2,\mu_2,v_2)\|_{L_2(U_2,H_2)}\leq L_{B_2}\|v_1-v_2\|_{H_2}
\nonumber\\\!\!\!\!\!\!\!\!&&~~~~~~~~~~~~~~~~~~~~~~~~~~~~~~~~~~~~~~~~~~~~~~~~~~~~~~
+C\big(\|u_1-u_2\|_{H_1}+\mathbb{W}_{2,H_1}(\mu_1,\mu_2)\big).~\label{58}
\end{eqnarray}
Furthermore, we also assume that
the smallest eigenvalue $\lambda_1$ of map $-\Delta$ satisfies
\begin{equation}\label{59}
\lambda_1-L_g-L_{B_2}^2>0.
\end{equation}

Hence, according to Theorem \ref{Th1} and \ref{main result 1}, we have the following result for the slow-fast distribution dependent stochastic porous media equation.
\begin{theorem}\label{SPME}
Assume that (\ref{55})-(\ref{59}) hold and $\Psi$ fulfills the conditions $(\Psi1)$-$(\Psi4)$ above.
Then for any initial values $x\in H_1$, $y\in H_2$ and $T>0$, system~(\ref{ex1}) has a unique solution $(X^{\varepsilon}_t,Y^{\varepsilon}_t)_{t\in[0,T]}$ such that
\begin{align*}
\mathbb{E} \left(\sup_{t\in[0,T]}\|X_{t}^{\vare}-\bar{X}_{t}\|_{H_1}^{2} \right)\leq C_T(1+\|x\|_{H_1}^2+\|y\|_{H_2}^2)\varepsilon^{1/3}\rightarrow0,~~~\text{as}~\varepsilon\rightarrow0,\label{2.2}
\end{align*}
where $C_T$ is a constant only depending on $T$,  and $\bar{X}_t$ is the solution of the corresponding averaged equation.
\end{theorem}

\begin{Rem}
(i) In \cite{HLL3,HL}, the authors have established the well-posedness and large deviation principle for Mckean-Vlasov stochastic porous media equations.
 To the best of our knowledge, there is no result on the averaging principle in the literature obtained for two-time-scale Mckean-Vlasov SPDE such as stochastic porous media equations here and  stochastic $p$-Laplace equations below.

 (ii) In \cite{LRSX1}, the authors have established the averaging principle result for classical (i.e. distribution independent) stochastic quasilinear SPDEs with slow and fast time-scales. In comparison to \cite{LRSX1}, we not only extend the corresponding averaging principle result to the distribution dependent case, but also explicitly obtain the strong convergence rate for the system in this work.
\end{Rem}

\subsection{Slow-fast  Mckean-Vlasov  stochastic $p$-Laplace equations}\label{laplace}
Now we apply our main results to establish the averaging principle for following  slow-fast Mckean-Vlasov stochastic $p$-Laplace equations
\begin{equation}\label{ex2}
\left\{ \begin{aligned}
&dX^{\varepsilon}_t=\left[div(|\nabla X^\varepsilon_t|^{p-2}\nabla X^\varepsilon_t)
+f(X^{\varepsilon}_t,\mathscr{L}_{X^{\varepsilon}_t},Y^{\varepsilon}_t)\right]dt
+B_1(X^{\varepsilon}_t,\mathscr{L}_{X^{\varepsilon}_t})dW^1_t,\\
&dY^{\varepsilon}_t=\frac{1}{\varepsilon}[\Delta Y^{\varepsilon}_t
+g(X^{\varepsilon}_t,\mathscr{L}_{X^{\varepsilon}_t},Y^{\varepsilon}_t)]dt+\frac{1}{\sqrt{\varepsilon}}B_2(X^{\varepsilon}_t,\mathscr{L}_{X^{\varepsilon}_t},Y^{\varepsilon}_t)d W^{2}_{t},
\\&X^{\varepsilon}_0=x, Y^{\varepsilon}_0=y.
\end{aligned} \right.
\end{equation}

For any $p\geq 2$, we set the following Gelfand triple for the slow equation
$$V_1:=W_0^{1,p}(\Lambda)\subset H_1:=  L^2(\Lambda) \subset  V_1^*,$$
and the following Gelfand triple for the fast equation
$$V_2:=W_0^{1,2}(\Lambda)\subset H_2:= L^{2}(\Lambda) \subset  V_2^*.$$

Denote $\bar{A}_1(u):=div(|\nabla u|^{p-2}\nabla u)$, which is called $p\text{-}Laplacian$ operator.
It is well-known that the operator $\bar{A}_1$ satisfies $({\mathbf{A}}{\mathbf{1}})$-$({\mathbf{A}}{\mathbf{4}})$, interested readers can refer to e.g.~\cite[Example 4.1.9]{LR1} for the detailed proof. Thus, according to Theorem \ref{Th1} and \ref{main result 1}, we have the following result for the slow-fast distribution dependent stochastic $p$-Laplace equations.
\begin{theorem}\label{laplace}
Assume that (\ref{55})-(\ref{59}) hold, then for any initial values $x\in H_1$, $y\in H_2$ and $T>0$, system~(\ref{ex2}) has a unique solution $(X^{\varepsilon}_t,Y^{\varepsilon}_t)_{t\in[0,T]}$ such that
\begin{align*}
\mathbb{E} \left(\sup_{t\in[0,T]}\|X_{t}^{\vare}-\bar{X}_{t}\|_{H_1}^{2} \right)\leq C_T(1+\|x\|_{H_1}^2+\|y\|_{H_2}^2)\varepsilon^{1/3}\rightarrow0,~~~\text{as}~\varepsilon\rightarrow0,\label{2.2}
\end{align*}
where $C_T$ is a constant only depending on $T$,  and $\bar{X}_t$ is the solution of the corresponding averaged equation.
\end{theorem}

\begin{Rem}
In particular, if we take $p=2$, $\bar{A}$ reduces to the classical Laplace operator. Therefore, our result above also covers some slow-fast distribution dependent semilinear SPDEs.
\end{Rem}

\subsection{Slow-fast Mckean-Vlasov SDEs}
Besides the above Mckean-Vlasov SPDEs, our main results are also applicable to  Mckean-Vlasov SDE models. For instance, we consider $V_i=H_i=\mathbb{R}^d$ ($i=1,2$) with the Euclidean norm $|\cdot|$ and inner product $\langle\cdot,\cdot\rangle$,
\begin{equation}\label{ex3}
\left\{ \begin{aligned}
&dX^{\varepsilon}_t=b_1(X^{\varepsilon}_t,\mathscr{L}_{X^{\varepsilon}_t},Y^{\varepsilon}_t)dt
+\sigma_1(X^{\varepsilon}_t,\mathscr{L}_{X^{\varepsilon}_t})dW^1_t,\\
&dY^{\varepsilon}_t=\frac{1}{\varepsilon}b_2(X^{\varepsilon}_t,\mathscr{L}_{X^{\varepsilon}_t},Y^{\varepsilon}_t)\dt+\frac{1}{\sqrt{\varepsilon}}\sigma_2(X^{\varepsilon}_t,\mathscr{L}_{X^{\varepsilon}_t},Y^{\varepsilon}_t)d W^{2}_{t},
\\&X^{\varepsilon}_0=x, Y^{\varepsilon}_0=y.
\end{aligned} \right.
\end{equation}
Suppose the coefficients
$$b_i: \mathbb{R}^d\times\mathscr{P}_2(\mathbb{R}^d)\times\mathbb{R}^d\rightarrow \mathbb{R}^d,~\sigma_1:\mathbb{R}^d\times\mathscr{P}_2(\mathbb{R}^d)\rightarrow \mathbb{R}^{d\times d}, ~\sigma_2:\mathbb{R}^d\times\mathscr{P}_2(\mathbb{R}^d)\times\mathbb{R}^d\rightarrow \mathbb{R}^{d\times d}$$
 are measurable and satisfy the following conditions (here $\mathbb{R}^{d\times d}$ denotes the set of real $d\times d$ matrices).
\begin{hypothesis}\label{h4}
For all $u,v,u_1,u_2,v_1,v_2,w\in \mathbb{R}^d$ and $\mu,\nu\in\mathscr{P}_2(\mathbb{R}^d)$.
\begin{enumerate}
\item [$({\mathbf{C}}{\mathbf{1}})$] There exists some constants $C,\kappa>0$ such that
    \begin{eqnarray*}|b_1(u_1,\mu,v_1)-b_1(u_2,\nu,v_2)|
\leq C\big(|u_1-u_2|+|v_1-v_2|+\mathbb{W}_{2,\mathbb{R}^d}(\mu,\nu)\big).
\end{eqnarray*}
Moreover,
\begin{eqnarray*}
\langle b_2(u_1,\mu,v_1)-b_2(u_2,\nu,v_2),v_1-v_2\rangle \leq -\kappa|v_1-v_2|^2+C\left(|u_1-u_2|^2+\mathbb{W}_{2,\mathbb{R}^d}(\mu,\nu)^2\right).
\end{eqnarray*}

\item [$({\mathbf{C}}{\mathbf{2}})$] There are some constants $L_{B_2},C>0$ such that
\begin{equation*}
\|\sigma_1(u,\mu)-\sigma_1(v,\nu)\|\leq C\big(|u-v|+\mathbb{W}_{2,\mathbb{R}^d}(\mu,\nu)\big),
\end{equation*}
and
\begin{equation*}
\|\sigma_2(u_1,\mu,v_1)-\sigma_2(u_2,\nu,v_2)\|\leq L_{\sigma_2}|u_1-u_2|+C\big(|v_1-v_2|+\mathbb{W}_{2,\mathbb{R}^d}(\mu,\nu)\big),
\end{equation*}
where $\|\cdot\|$ denotes the matrix norm.
\end{enumerate}
\end{hypothesis}

By Theorem \ref{Th1} and \ref{main result 1}, we can derive the averaging principle for the
slow-fast Mckean-Vlasov SDEs.
\begin{theorem}\label{SDE}
Assume that  Hypothesis \ref{h4} hold and $\kappa>2L_{\sigma_2}^2$,
then for any initial values $x,y\in\mathbb{R}^d$ and $T>0$, system~(\ref{ex3}) has a unique solution $(X^{\varepsilon}_t,Y^{\varepsilon}_t)_{t\in[0,T]}$ such that
\begin{align*}
\mathbb{E} \left(\sup_{t\in[0,T]}|X_{t}^{\vare}-\bar{X}_{t}|^{2} \right)\leq C_T(1+\|x\|_{H_1}^2+\|y\|_{H_2}^2)\varepsilon^{1/3}\rightarrow0,~~~\text{as}~\varepsilon\rightarrow0,
\end{align*}
where $C_T$ is a constant only depending on $T$,  and $\bar{X}_t$ is the solution of the corresponding averaged equation.
\end{theorem}

\begin{Rem} Using the techniques of time discretization and Poisson equation, R\"{o}ckner et al.~\cite{RSX} established the optimal strong convergence rate $1/2$ of averaging principle for two-time-scale McKean-Vlasov SDEs under some fairly strong conditions, such as the regularity
of first-order and second-order partial derivatives of the coefficients. The convergence rate obtained here is not optimal, since we only assume
the coefficients satisfy some monotonicity and Lipschitz conditions, which is in general much weaker than the assumptions in~\cite{RSX}. Moreover, our main results are not only covering this type of models, but also applicable to various two-time-scale McKean-Vlasov (nonlinear) SPDEs.
\end{Rem}

\noindent\textbf{Acknowledgements} {The authors would like to thank the anonymous referee
for valuable suggestions and thank Prof.~Feng-Yu Wang for helpful discussions. The research
of W. Hong is supported by NSFC (No.~12171354). The research
of S. Li is supported by NSFC (No.~12001247),
 NSF of Jiangsu Province (No.~BK20201019),  NSF of Jiangsu Higher Education Institutions of China (No. 20KJB110015)
and the Foundation of Jiangsu Normal University (No.~19XSRX023). The research of W. Liu is supported by NSFC (No.~12171208, 11822106, 11831014, 12090011) and the PAPD of Jiangsu Higher Education Institutions.}

\end{document}